\documentclass[11pt,twoside,reqno]{amsart}

\usepackage{amssymb, latexsym}
\usepackage[italian,english]{babel}
\usepackage{amsmath,amsfonts,amsthm}
\usepackage{paralist}
\usepackage[top=2.50cm, bottom=2.50cm, left=2.50cm, right=2.50cm]{geometry}

\usepackage{xspace}
\usepackage[bookmarksnumbered,colorlinks]{hyperref}
\usepackage{graphics}

\usepackage{color}

\numberwithin{equation}{section}

\usepackage{xspace}
\usepackage[bookmarksnumbered,colorlinks]{hyperref}
\usepackage{graphics}
\def\bb#1\eb{\textcolor{blue}
{#1}} %
\def\br#1\er{\textcolor{red}
{#1}} %
\def\bv#1\ev{\textcolor{green}
{#1}} %
\def\bc#1\ec{\textcolor{cyan}
{#1}} %

\usepackage{graphics}

\def\Xint#1{\mathchoice
  {\XXint\displaystyle\textstyle{#1}}%
  {\XXint\textstyle\scriptstyle{#1}}%
  {\XXint\scriptstyle\scriptscriptstyle{#1}}%
  {\XXint\scriptscriptstyle\scriptscriptstyle{#1}}%
  \!\int}
\def\XXint#1#2#3{{\setbox0=\hbox{$#1{#2#3}{\int}$}
  \vcenter{\hbox{$#2#3$}}\kern-.5\wd0}}
\def\-int{\Xint -}

\newcommand{\e}{\varepsilon}
\newcommand{\R}{\mathbb{R}}
\newcommand{\N}{\mathcal{N}}
\newcommand{\M}{\mathcal{M}}
\newcommand{\A}{\mathcal{A}}

\newcommand{\X}{\mathbb{X}}
\newcommand{\J}{\mathcal{J}}
\newcommand{\I}{\mathcal{I}}

\DeclareMathOperator{\supp}{supp}
\DeclareMathOperator{\B}{\mathcal{B}}

\DeclareMathOperator{\h}{\mathcal{W}_{\e}}
\newcommand{\p}{p^{*}_{s}}

\newtheorem{prop}{Proposition}[section]
\newtheorem{lem}{Lemma}[section]
\newtheorem{thm}{Theorem}[section]

\newtheorem{remark}{Remark}[section]

\begin{document}
\title[$p$-Fractional Choquard equation]{On the multiplicity and concentration of positive solutions for a $p$-fractional Choquard equation in $\R^{N}$}
\author[V. Ambrosio]{Vincenzo Ambrosio}
\address{Vincenzo Ambrosio\hfill\break\indent 
Dipartimento di Ingegneria Industriale e Scienze Matematiche \hfill\break\indent
Universit\`a Politecnica delle Marche\hfill\break\indent
Via Brecce Bianche, 12\hfill\break\indent
60131 Ancona (Italy)}
\email{ambrosio@dipmat.univpm.it}

\keywords{Fractional Choquard equation; fractional $p$-Laplacian; variational methods}
\subjclass[2010]{35A15, 35R11, 45G05}

\date{}

\begin{abstract}
In this paper we deal with the following fractional Choquard equation
\begin{equation*}
\left\{
\begin{array}{ll}
\varepsilon^{sp}(-\Delta)^{s}_{p} u + V(x)|u|^{p-2}u = \varepsilon^{\mu-N}\left(\frac{1}{|x|^{\mu}}*F(u)\right)f(u) \mbox{ in } \mathbb{R}^{N},\\
u\in W^{s,p}(\mathbb{R}^{N}), \quad u>0 \mbox{ in } \mathbb{R}^{N},
\end{array}
\right.
\end{equation*}
where $\varepsilon>0$ is a small parameter, $s\in (0, 1)$, $p\in (1, \infty)$, $N>sp$, $(-\Delta)^{s}_{p}$ is the fractional $p$-Laplacian, $V$ is a positive continuous potential, $0<\mu<sp$, and $f$ is a continuous superlinear function with subcritical growth.
Using minimax arguments and the Ljusternik-Schnirelmann category theory, we obtain the existence, multiplicity and concentration of positive solutions for $\varepsilon>0$ small enough.
\end{abstract}

\maketitle

\section{Introduction}

\noindent

In this paper we focus our attention on the following nonlinear fractional Choquard equation
\begin{equation}\label{P}
\left\{
\begin{array}{ll}
\e^{sp}(-\Delta)^{s}_{p} u + V(x)|u|^{p-2}u = \e^{\mu-N}\left(\frac{1}{|x|^{\mu}}*F(u)\right)f(u) \mbox{ in } \R^{N},\\
u\in W^{s,p}(\R^{N}), \quad u>0 \mbox{ in } \R^{N},
\end{array}
\right.
\end{equation}
where $\e>0$ is a parameter, $s\in (0,1)$, $p\in (1, \infty)$, $N> sp$, $0<\mu<sp$ and $V:\R^{N}\rightarrow \R$ and $f:\R\rightarrow \R$ are continuous functions. 
The fractional $p$-Laplacian operator $(-\Delta)^{s}_{p}$  is defined, up to normalization factors, for any $u:\R^{N}\rightarrow \R$ smooth enough
by
$$
(-\Delta)^{s}_{p}u(x)=2\lim_{r\rightarrow 0}\int_{\R^{N}\setminus \mathcal{B}_{r}(x)} \frac{|u(x)-u(y)|^{p-2}(u(x)-u(y))}{|x-y|^{N+sp}} dy \quad (x\in \R^{N}).
$$
The above operator is a nonlocal version of the classical $p$-Laplacian $\Delta_{p}$ and it is an extension of the fractional Laplacian $(-\Delta)^{s}$ (that is $p=2$); see \cite{DPV, KP} for more details.

When $s=1$, $p=2$, $V(x)\equiv 1$, $\e=1$ and $F(u)=\frac{|u|^{2}}{2}$, we can see that \eqref{P} reduces to the Choquard-Pekar equation
\begin{equation}\label{CE}
-\Delta u + u = \left(\frac{1}{|x|^{\mu}}*|u|^{2}\right)u \mbox{ in } \R^{N}
\end{equation}
which goes back to the description of a polaron at rest in Quantum Field Theory by Pekar \cite{Pek} in 1954.  
In particular, when $u$ is a solution to \eqref{CE}, we can see that $\psi(x, t)=u(x)e^{-\imath t}$ is a solitary wave of the following Hartree equation
$$
\imath \frac{\partial \psi}{\partial t}=-\Delta \psi-\left(\frac{1}{|x|^{\mu}}*|\psi|^{2}\right)\psi \mbox{ in } \R^{3}\times \R_{+},
$$
introduced by Choquard in 1976 to describe an electron trapped in its own hole as approximation to Hartree-Fock Theory of one component plasma; see \cite{LL, Pen}.

From a mathematical point of view, equation \eqref{CE} and its generalizations have been widely investigated.
Lieb \cite{Lieb} proved the existence and  uniqueness, up to translations, of the ground state to \eqref{CE}. Lions \cite{Lions} obtained the existence of a sequence of radially symmetric solutions via critical point theory. Ma and Zhao \cite{MZ} showed the symmetry of positive solutions for a generalized Choquard equation. Moroz and Van Shaftingen \cite{MVS1} established regularity, radial symmetry and asymptotic behavior at infinity of positive solutions. For other interesting results on Choquard equations we refer to \cite{Ack, AY2, GY, MVS2, WW} and the survey \cite{MVS3}.
\noindent

Recently, the study of problems involving fractional and nonlocal operators has received a great interest in view of concrete real-world applications, such as phase transitions, anomalous diffusion, population dynamics, optimization, finance, conservation laws, ultra-relativistic limits of quantum mechanics, quasi-geostrophic flows and many others; see  \cite{DPV, MBRS}.
In particular, when $p=2$ in \eqref{P}, a large number of papers have been devoted to the study of fractional Schr\"odinger equations \cite{Laskin1} involving local nonlinearities; see for instance \cite{AM, A3, AH, DMV, FQT, Secchi1} and the references therein.\\
However, in the literature there are only few papers dealing with fractional Schr\"odinger equations like \eqref{P} in which the nonlocal term appears also in the nonlinearity. 
Frank and Lenzmann \cite{FL} proved analyticity and radial symmetry of positive ground state for the $L^{2}$ critical boson star equation
$$
(-\Delta)^{\frac{1}{2}}u+u=\left(\frac{1}{|x|^{\mu}}*|u|^{2}\right)u \mbox{ in } \R^{3}.
$$
d'Avenia et al. \cite{DSS} dealt with the regularity,  existence and non existence, symmetry and decay properties of solutions to
$$
(-\Delta)^{s}u+u=\left(\frac{1}{|x|^{\mu}}*|u|^{p}\right)|u|^{p-2}u \mbox{ in } \R^{N}.
$$
Shen et al. \cite{SGY} investigated the existence of ground state solutions for a fractional Choquard equation involving a general nonlinearity. 
In \cite{Apota} the author used penalization technique and  Ljusternik-Schnirelmann theory to study the multiplicity and concentration of positive solutions to 
$$
\e^{2s} (-\Delta)^{s}u+V(x) u=\e^{\mu-N}\left(\frac{1}{|x|^{\mu}}*F(u)\right)f(u) \mbox{ in } \R^{N}
$$
when the potential $V$ has a local minimum.\\
On the other hand,  in the last few years, a great attention has been focused on the study of fractional $p$-Laplacian problems. Indeed, from the mathematical point of view, the fractional $p$-Laplacian has taken relevance because two phenomena are present in it: the nonlinearity of the operator and its nonlocal character. 
For instance, fractional $p$-eigenvalue problems have been considered in \cite{FrP, LLq}, some interesting regularity results for weak solutions can be found in \cite{DCKP, IMS, P}, several existence and multiplicity results for problems set in bounded domains or in the whole of $\R^{N}$ have been established in \cite{A1, A4, AI, DPQna, FiP, MMB, PXZ1}, while $p$-fractional Choquard equations have been studied in \cite{BBMP, LR, PXZ2}.

Motivated by the above papers, in this work we aim to study the existence, multiplicity, and concentration of positive solutions for the fractional Choquard equation \eqref{P} assuming that 
the potential $V: \R^{N}\rightarrow \R$ verifies the following condition due to Rabinowitz \cite{Rab}:
\begin{compactenum}[$(V)$]
\item $V_{\infty}=\liminf_{|x|\rightarrow \infty} V(x)>V_{0}=\inf_{x\in \R^{N}} V(x)$,
\end{compactenum}
and the nonlinearity $f: \R\rightarrow \R$ is a continuous function satisfying the following hypotheses: 
\begin{compactenum}[($f_1$)]
\item there exist $C>0$ and $q_{1}, q_{2}>\frac{p}{2}$ with $\frac{p}{2}(2-\frac{\mu}{N})<q_{1}\leq q_{2}<\frac{\p}{2}(2-\frac{\mu}{N})$ such that 
$$
|f(t)|\leq C(|t|^{q_{1}-1}+|t|^{q_{2}-1}) \quad \forall t\in \R;
$$
\item there exists $\theta>p$ such that $0<\theta F(t)\leq 2f(t)t$ for all $t>0$, where $F(t)=\int_{0}^{t} f(\tau) d\tau$; 
\item $\displaystyle{t\mapsto \frac{f(t)}{t^{\frac{p}{2}-1}}}$ is increasing for $t>0$.
\end{compactenum}
The main result of this paper establishes  the existence of multiple positive solutions of (\ref{P}) involving the  Ljusternik-Schnirelmann category of the sets $M$ and $M_{\delta}$ defined as 
$$
M=\{x\in \Lambda: V(x)=V_{0}\} \mbox{ and } M_{\delta}=\{x\in \R^{N}: dist(x, M)\leq \delta\}, \mbox{ for } \delta>0.
$$
We recall that if $Y$ is a given closed set of a topological space $X$, we denote by $cat_{X}(Y)$ the Ljusternik-Schnirelmann category of $Y$ in $X$, that is the least number of closed and contractible sets in $X$ which cover $Y$; see \cite{MW, W}. 
More precisely, we are able to prove that
\begin{thm}\label{thmA1}
Suppose that $V$ verifies $(V)$,  $0<\mu<sp$ and $f$ satisfies $(f_1)$-$(f_3)$ with 
$$
p<q_{1}\leq q_{2}<\frac{p(N-\mu)}{N-sp}.
$$
Then, for any $\delta>0$, there exists $\e_{\delta}>0$ such that, for any $\e\in (0, \e_{\delta})$, problem \eqref{P} has at least $cat_{M_{\delta}}(M)$ positive solutions. Moreover, if $u_{\e}$ denotes one of these positive solutions and $x_{\e}\in \R^{N}$ its global maximum, then 
$$
\lim_{\e\rightarrow 0} V(x_{\e})=V_{0},
$$
and there exists $C>0$ such that
$$
0<u_{\e}(x)\leq \frac{C\e^{N+sp}}{\e^{N+sp}+|x-x_{\e}|^{N+sp}} \quad \forall x\in \R^{N}.
$$
\end{thm}

\noindent
We note that the restriction on the exponents $q_1$ and $q_2$ is justified by the following Hardy-Littlewood-Sobolev inequality: 
\begin{thm}\label{HLS}\cite{LL}
Let $r, t>1$ and $0<\mu<N$ such that $\frac{1}{r}+\frac{\mu}{N}+\frac{1}{t}=2$. Let $f\in L^{r}(\R^{N})$ and $h\in L^{t}(\R^{N})$. Then there exists a sharp constant $C(r, N, \mu, t)>0$ independent of $f$ and $h$ such that 
$$
\int_{\R^{N}}\int_{\R^{N}} \frac{f(x)h(y)}{|x-y|^{\mu}}\, dx dy\leq C(r, N, \mu, t)|f|_{r}|h|_{t}.
$$
\end{thm}
\noindent
Indeed, when we consider the model case $F(u)=|u|^{q}$, we can see that 
$$
\int_{\R^{N}} \left(\frac{1}{|x|^{\mu}}*F(u)\right) F(u)\, dx
$$
is well-defined if $F(u)\in L^{t}(\R^{N})$ for $t>1$ such that $\frac{2}{t}+\frac{\mu}{N}=2$. Then, using the fractional Sobolev embedding $W^{s,p}(\R^{N})\subset L^{r}(\R^{N})$ for all $r\in [p, \p]$, we have to require that $tq\in [p, \p]$ which together with the fact that we are considering the subcritical case, forces to suppose that 
$$
\frac{p}{2}\left(2-\frac{\mu}{N}\right)< q< \frac{\p}{2}\left(2-\frac{\mu}{N}\right).
$$
Here, we only consider the case $q>p$.

The proof of Theorem \ref{thmA1} is obtained applying suitable variational methods. 
We would like to note that our result improves and extends Theorem 1.4 in \cite{AY2} in the fractional setting because here we assume that $f$ is only continuous.
Indeed, differently from \cite{AY2}, we cannot apply standard Nehari manifolds arguments to study \eqref{P} due to the fact that the Nehari manifold associated with \eqref{P} is not  differentiable. To overcome this difficulty, we use some variants of critical point theorems due to Szulkin and Weth \cite{SW}; see Sections $3$ and $4$.
We also emphasized that the presence of the fractional $p$-Laplacian operator and the convolution term, both nonlocal operators, make our study more complicated with respect to \cite{AY2}, and a more accurate inspection will be done; see Section $2$ for some useful technical results. 

Indeed, after proving that the levels of compactness are strongly related to the behavior of the potential $V(x)$ at infinity (see Proposition \ref{prop2.1}), we are able to deduce the existence of a ground state solution for  \eqref{P} provided that $\e>0$ is sufficiently small.
Then, we obtain multiple solutions by using a technique due to Benci and Cerami \cite{BC}. The main ingredient is to make precisely comparisons between the category of some sublevel sets of the energy functional associated with \eqref{P} and the category of the set $M$.
We also investigate the concentration of positive solutions $u_{\e}$ of \eqref{P}. More precisely, we combine a Moser iteration technique \cite{Moser} with the H\"older regularity results obtained for $(-\Delta)^{s}_{p}$ (see \cite{DCKP, IMS}) to deduce that $u_{\e}(x)$ decays at zero as $|x|\rightarrow \infty$ uniformly in $\e$. This information together with the continuity of $V$ will be fundamental to infer that $u_{\e}$ concentrates around global minimum of the potential $V$ as $\e \rightarrow 0$.  Moreover, arguing as in \cite{AI, BBMP} and using the recent result in \cite{DPQna}, we find out that the solutions of $\eqref{P}$ have a power-type decay at infinity; see at the end of Section $6$. \\
As far as we know, there are no results in the literature concerning the multiplicity and concentration of positive  solutions to \eqref{P} using the Ljusternik-Schnirelmann category theory, so the goal of the present paper is to give a first result in this direction.

\medskip
\noindent
The body of the paper is the following. In Section $2$ we collect some lemmas which will be useful along the paper. In Section $3$ we outline the variational framework for studying \eqref{P}. Section $4$ is devoted to the study of the autonomous problem associated with \eqref{P}. In Section $5$ we provide a first existence result for \eqref{P}. The last section focuses on the multiplicity and concentration of solutions to \eqref{P}.

Before concluding this introduction, we would like to recall that using the change of variable  $u(x)\mapsto u(\e x)$ we can see that problem (\ref{P}) is equivalent to the following one
\begin{equation}\label{R}
\left\{
\begin{array}{ll}
(-\Delta)^{s}_{p} u + V(\e x)|u|^{p-2}u =  \left(\frac{1}{|x|^{\mu}}*F(u)\right)f(u)  \mbox{ in } \R^{N},  \\
u\in W^{s,p}(\R^{N}), \quad u>0 \mbox{ in } \R^{N},
\end{array}
\right.
\end{equation}
which will be considered in the next sections.

\section{Preliminaries}

In this section we introduce the notations and we give some lemmas which we will use later.

Fix $s\in (0,1)$ and $p\in (1, \infty)$. We denote by $\mathcal{D}^{s, p}(\R^{N})$ the completion of $C^{\infty}_{c}(\R^{N})$ with respect to the Gagliardo seminorm
$$
[u]^{p}_{s,p}=\iint_{\R^{2N}} \frac{|u(x)-u(y)|^{p}}{|x-y|^{N+sp}} \, dx \, dy,
$$
or equivalently
$$
\mathcal{D}^{s, p}(\R^{N})=\left\{u\in L^{p^{*}_{s}}(\R^{N}): [u]_{s,p}<\infty\right\}.
$$
Let us define the fractional Sobolev space
$$
W^{s,p}(\R^{N})= \left\{u\in L^{p}(\R^{N}) : \frac{|u(x)-u(y)|}{|x-y|^{\frac{N+sp}{p}}} \in L^{p}(\R^{2N}) \right \}
$$
endowed with the natural norm 
$$
\|u\|^{p}_{s,p} = [u]^{p}_{s,p} + |u|_{p}^{p}.
$$

\noindent
For reader's convenience, we recall the following fundamental embeddings:
\begin{thm}\cite{DPV}\label{Sembedding}
Let $s\in (0,1)$ and $p\in (1, \infty)$ be such that $N>sp$. Then there exists a constant $S_{*}=S(N, s, p)>0$
such that for any $u\in \mathcal{D}^{s,p}(\R^{N})$
\begin{equation}\label{FSI}
|u|^{p}_{\p} \leq S_{*} [u]^{p}_{s,p}. 
\end{equation}
Moreover, $W^{s,p}(\R^{N})$ is continuously embedded in $L^{q}(\R^{N})$ for any $q\in [p, p^{*}_{s}]$ and compactly in $L^{q}_{loc}(\R^{N})$ for any $q\in [p, p^{*}_{s})$. 
\end{thm}

\noindent
We also have the following Lions-type lemma.
\begin{lem}\label{Lions}
Let $N>sp$. If $(u_{n})$ is a bounded sequence in $W^{s,p}(\R^{N})$ and if
$$
\lim_{n \rightarrow \infty} \sup_{y\in \R^{N}} \int_{\B_{R}(y)} |u_{n}|^{p} dx=0
$$
where $R>0$,
then $u_{n}\rightarrow 0$ in $L^{t}(\R^{N})$ for all $t\in (p, p^{*}_{s})$.
\end{lem}
\begin{proof}
Let $\tau \in (r, \p)$. By the H\"older and Sobolev inequality we have for all $n\in \mathbb{N}$
\begin{align*}
|u_{n}|_{L^{\tau}(\B_{R}(y))} &\leq |u_{n}|_{L^{r}(\B_{R}(y))}^{1-\alpha} |u_{n}|_{L^{\p}(\B_{R}(y))}^{\alpha} \\
&\leq C |u_{n}|_{L^{r}(\B_{R}(y))}^{1- \alpha} \|u_{n}\|_{s,p}^{\alpha}
\end{align*}
where $\alpha= \frac{\tau-r}{\p-r}\frac{\p}{\tau}$.
Now, covering $\R^{N}$ by balls of radius $R$, in such a way that each point of $\R^{N}$ is contained in at most $N+1$ balls, we find
\begin{align*}
|u_{n}|_{\tau}^{\tau} \leq C \sup_{y\in \R^{N}} \left(\int_{\B_{R}(y)} |u_{n}|^{r} dx \right)^{(1-\alpha)\tau} \|u_{n}\|^{\alpha \tau}_{s,p}.
\end{align*}
From the assumption and the boundedness of $(u_{n})$ in $W^{s, p}(\R^{N})$, we obtain that $u_{n}\rightarrow 0$ in $L^{\tau}(\R^{N})$. Using an interpolation argument we get the thesis.
\end{proof}

Next, we collect some technical results which will be very useful along the paper.
\begin{lem}\label{pp}
Let $u\in W^{s, p}(\R^{N})$ and $\phi\in C^{\infty}_{c}(\R^{N})$ be such that $0\leq \phi\leq 1$, $\phi=1$ in $\B_{1}(0)$ and $\phi=0$ in $\R^{N} \setminus \B_{2}(0)$. Set $\phi_{r}(x)=\phi(\frac{x}{r})$.  Then
$$
\lim_{r\rightarrow \infty} [u \phi_{r}-u]_{s, p}=0 \quad \mbox{ and } \quad \lim_{r\rightarrow \infty} |u\phi_{r}-u|_{p}=0.
$$
\end{lem}
\begin{proof}
Since $\phi_{r}u\rightarrow u$ a.e. in $\R^{N}$ as $r\rightarrow \infty$, $0\leq \phi\leq 1$ and $u\in L^{p}(\R^{N})$, we can apply the Dominated Convergence Theorem to see that $\lim_{r\rightarrow \infty} |u\phi_{r}-u|_{p}=0$.
Now, we show the first relation of limit. Let us note that 
\begin{align*}
[u\phi_{r}-u]^{p}_{s,p}&\leq 2^{p-1}\left[\iint_{\R^{2N}} \frac{|u(x)-u(y)|^{p}|\phi_{r}(x)-1|^{p}}{|x-y|^{N+sp}}dx dy+\iint_{\R^{2N}} \frac{|\phi_{r}(x)-\phi_{r}(y)|^{p}|u(y)|^{p}}{|x-y|^{N+sp}}dx dy\right] \\
&=: 2^{p-1} [A_{r}+B_{r}].
\end{align*}
Invoking the Dominated Convergence Theorem it is easy to deduce that $A_{r}\rightarrow 0$ as $r\rightarrow \infty$.
Concerning $B_{r}$, recalling that $0\leq \phi_{r}\leq 1$, $|\nabla \phi_{r}|_{\infty}\leq C/r$ and using the polar coordinates, we obtain
\begin{align*}
&\iint_{\R^{2N}} \frac{|\phi_{r}(x)-\phi_{r}(y)|^{p}}{|x-y|^{N+sp}}|u(y)|^{p} dx dy \\
&=\int_{\R^{N}} \int_{|x-y|>r} \frac{|\phi_{r}(x)-\phi_{r}(y)|^{p}}{|x-y|^{N+sp}}|u(u)|^{p} dy dx +\int_{\R^{N}} \int_{|x-y|\leq r} \frac{|\phi_{r}(x)-\phi_{r}(y)|^{p}}{|x-y|^{N+sp}}|u(y)|^{p} dy dx \\
&\leq C \int_{\R^{N}} |u(y)|^{p}  \left(\int_{|x-y|>r} \frac{dx}{|x-y|^{N+sp}}\right) dy + \frac{C}{r^{p}} \int_{\R^{N}} |u(y)|^{p} \left(\int_{|x-y|\leq r} \frac{dx}{|x-y|^{N+sp-p}}\right) dy  \\
&\leq C \int_{\R^{N}} |u(y)|^{p}  \left(\int_{|z|>r} \frac{dz}{|z|^{N+sp}}\right) dy + \frac{C}{r^{p}} \int_{\R^{N}} |u(y)|^{p} \left(\int_{|z|\leq r} \frac{dz}{|z|^{N+sp-p}}\right) dy  \\
&\leq C \int_{\R^{N}} |u(y)|^{p} dy \left(\int_{r}^{\infty} \frac{d\rho}{\rho^{sp+1}}\right)  + \frac{C}{r^{p}} \int_{\R^{N}} |u(y)|^{p} dy \left(\int_{0}^{r} \frac{d\rho}{\rho^{sp-p+1}}\right)  \\
&\leq  \frac{C}{r^{sp}} \int_{\R^{N}} |u(y)|^{p} dy+\frac{C}{r^{p}} r^{-sp+p}\int_{\R^{N}} |u(y)|^{p} dy \\
&\leq  \frac{C}{r^{sp}} \int_{\R^{N}} |u(y)|^{p} dy\leq \frac{C}{r^{sp}}\rightarrow 0 \quad \mbox{ as } \quad r\rightarrow \infty. 
\end{align*}
Hence, $B_{r}\rightarrow 0$ as $r\rightarrow \infty$ and  we can conclude the proof of lemma.
\end{proof}

\begin{lem}\label{lemVince}
Let $w\in \mathcal{D}^{s, p}(\R^{N})$ and $(z_{n})\subset \mathcal{D}^{s, p}(\R^{N})$ be a sequence such that $z_{n}\rightarrow 0$ a.e. and $[z_{n}]_{s,p}\leq C$ for any $n\in \mathbb{N}$. Then we have
\begin{align*}
\iint_{\R^{2N}} |\mathcal{A}(z_{n} + w) - \mathcal{A}(z_{n}) - \mathcal{A}(w)|^{p'} dx= o_{n}(1),
\end{align*}
where $\mathcal{A}(u):=\frac{|u(x)- u(y)|^{p-2}(u(x)- u(y))}{|x-y|^{\frac{N+sp}{p'}}}$ and $p'= \frac{p}{p-1}$ is the conjugate exponent of $p$.
\end{lem}
\begin{proof}
Firstly, we deal with the case $p\geq 2$. Using the Mean Value Theorem, the Young inequality and $p\geq 2$, we can see that for fixed $\e>0$ there exists $C_{\e}>0$ such that
\begin{equation}\label{abp-2}
||a+b|^{p-2}(a+b)-|a|^{p-2}a|\leq \e |a|^{p-1}+C_{\e}|b|^{p-1} \quad \mbox{ for all } a, b\in \R.
\end{equation}
Taking
$$
a=\frac{z_{n}(x)-z_{n}(y)}{|x-y|^{\frac{N+sp}{p}}} \quad \mbox{ and } \quad b=\frac{w(x)-w(y)}{|x-y|^{\frac{N+sp}{p}}}
$$
in \eqref{abp-2} we obtain
\begin{align*}
&\Bigl|\frac{|(z_{n}(x)+w(x))- (z_{n}(y)+w(y))|^{p-2}((z_{n}(x)+w(x))- (z_{n}(y)+w(y)))}{|x-y|^{\frac{N+sp}{p'}}} \\
&\qquad- \frac{|z_{n}(x)- z_{n}(y)|^{p-2}(z_{n}(x)- z_{n}(y))}{|x-y|^{\frac{N+sp}{p'}}} \Bigr| \\
&\leq \e \frac{|z_{n}(x)-z_{n}(y)|^{p-1}}{|x-y|^{\frac{N+sp}{p'}}}+ C_{\e}  \frac{|w(x)-w(y)|^{p-1}}{|x-y|^{\frac{N+sp}{p'}}}.
\end{align*}
Let us consider the function $H_{\e, n}:\R^{2N}\rightarrow \R_{+}$ defined by 
\begin{align*}
H_{\e, n} (x, y):= \max \left\{ |\A(z_{n} + w) - \A(z_{n}) - \A(w)|- \e \frac{|z_{n}(x)-z_{n}(y)|^{p-1}}{|x-y|^{\frac{N+sp}{p'}}}, 0 \right\}.
\end{align*}
We can see that $H_{\e, n} \rightarrow 0$ a.e. in $\R^{2N}$ as $n\rightarrow \infty$, and
\begin{align*}
0\leq H_{\e, n}(x, y) \leq C_{1} \frac{|w(x)-w(y)|^{p-1}}{|x-y|^{\frac{N+sp}{p'}}}\in L^{p'}(\R^{2N}).
\end{align*}
Then, using the Dominated Convergence Theorem, we have
\begin{align*}
\int_{\R^{2N}} |H_{\e, n} |^{p'} dxdy \rightarrow 0 \quad \mbox{ as } n\rightarrow \infty.
\end{align*}
On the other hand, from the definition of $H_{\e, n}$, we deduce that
\begin{align*}
|\A(z_{n} + w) - \A(z_{n}) - \A(w)|\leq \e \frac{|z_{n}(x)-z_{n}(y)|^{p-1}}{|x-y|^{\frac{N+sp}{p'}}} + H_{\e, n},
\end{align*}
so we obtain
\begin{align*}
|\A(z_{n} + w) -\A(z_{n}) - \A(w)|^{p'}\leq C_{2}\left[\e^{p'} \frac{|z_{n}(x)-z_{n}(y)|^{p}}{|x-y|^{N+sp}} + (H_{\e, n})^{p'}\right].
\end{align*}
Therefore
\begin{align*}
\limsup_{n\rightarrow \infty} &\iint_{\R^{2N}} |\A(z_{n} + w) - \A(z_{n}) - \A(w)|^{p'} dxdy \leq C_{2}\e^{p'}\limsup_{n\rightarrow \infty} [z_{n}]^{p}_{s,p} \leq C_{3}\e^{p'},
\end{align*}
and by the arbitrariness of $\e>0$ we get the thesis.\\
Now, we suppose that $1<p<2$. Using Lemma $3.1$ in \cite{MeW}, we know that
$$
\sup_{c\in \R^{N}, d\neq 0} \left|\frac{|c+d|^{p-2}(c+d)-|c|^{p-2}c}{|d|^{p-1}}\right|<\infty.
$$
Taking
$$
c=\frac{z_{n}(x)-z_{n}(y)}{|x-y|^{\frac{N+sp}{p}}} \quad \mbox{ and } \quad d=\frac{w(x)-w(y)}{|x-y|^{\frac{N+sp}{p}}},
$$
we can conclude the proof applying the Dominated Convergence Theorem.
\end{proof}

\noindent
\noindent
In order to study \eqref{R}, for any $\e>0$, we introduce the following fractional Sobolev space
$$
\h=\left\{u\in W^{s,p}(\R^{N}): \int_{\R^{N}} V(\e x) |u(x)|^{p}\, dx<\infty \right\}
$$
endowed with the norm
$$
\|u\|^{p}_{\e}=[u]_{s,p}^{p}+\int_{\R^{N}} V(\e x) |u(x)|^{p}\, dx.
$$
In view of assumption $(V)$ and Theorem \ref{Sembedding}, it is easy to check that the following result holds.
\begin{thm}\label{embedding}
The space $\h$ is continuously embedded in $W^{s, p}(\R^{N})$.
Therefore, $\h$ is continuously embedded in $L^{r}(\R^{N})$ for any $r\in [p, \p]$ and compactly embedded in $L^{r}_{loc}(\R^{N})$ for any $r\in [1, \p)$.
\end{thm}

When the potential $V(x)$ is coercive, we can obtain the compactness of $\h$ into the Lebesgue spaces $L^{r}(\R^{N})$.
\begin{thm}\label{Cheng}
Let $V_{\infty}=\infty$. Then $\h$ is compactly embedded into $L^{r}(\R^{N})$ for any $r\in [p, \p)$.
\end{thm}
\begin{proof}
Firstly, we assume that $r=p$. From Theorem \ref{embedding} we know that $\h\subset L^{p}(\R^{N})$. Let $(u_{n})$ be a sequence such that $u_{n}\rightharpoonup 0$ in $\h$. Then, $u_{n}\rightharpoonup 0$
 in $W^{s, p}(\R^{N})$.\\
Let us define
\begin{equation}\label{ter1}
M:= \sup_{n\in \mathbb{N}} \|u_{n}\|_{\e} <\infty.
\end{equation}
Since $V$ is coercive, for any $\eta>0$ there exists $R=R_{\eta}>0$ such that
\begin{equation}\label{ter2}
\frac{1}{V(\e x)}<\eta, \quad \mbox{ for any } |x|>R.
\end{equation}
Since $u_{n}\rightarrow 0$ in $L^{p}(\B_{R}(0))$, there exists $n_{0}>0$ such that
\begin{equation}\label{ter3}
\int_{\B_{R}(0)} |u_{n}|^{p} dx \leq \eta \quad \mbox{ for any } n\geq n_{0}.
\end{equation}
Hence, for any $n\geq n_{0}$, by \eqref{ter1}-\eqref{ter3}, we have
\begin{align*}
\int_{\R^{N}} |u_{n}|^{p} dx &=\int_{\B_{R}(0)} |u_{n}|^{p} dx + \int_{\B_{R}^{c}(0)} |u_{n}|^{p} dx \\
&< \eta +\eta \int_{\B_{R}^{c}(0)} V(\e x) |u_{n}|^{p} dx \leq \eta(1+ M^{p}).
\end{align*}
Therefore, $u_{n}\rightarrow 0$ in $L^{p}(\R^{N})$.\\
When $r>p$, using the conclusion of $r=p$, interpolation inequality and Theorem \ref{Sembedding}, we can see that
\begin{equation*}
|u_{n}|_{r} \leq C [u_{n}]^{\alpha}_{s,p} |u_{n}|_{p}^{1-\alpha},
\end{equation*}
where $\frac{1}{r}=\frac{\alpha}{p}+\frac{1-\alpha}{\p}$, which yields the conclusion as required.
\end{proof}

\noindent
Let us prove the following splitting for the $\h$-norm in the spirit of the Brezis-Lieb Lemma \cite{BL}.
\begin{lem}\label{lemPSY}
If $(u_{n})$ is a bounded sequence in $\h$, then
\begin{align*}
\|u_{n}- u\|_{\e}^{p} = \|u_{n}\|_{\e}^{p} - \|u\|_{\e}^{p} + o_{n}(1).
\end{align*}
\end{lem}
\begin{proof}
From the Brezis-Lieb Lemma \cite{BL} we know that if $r\in (1, \infty)$ and $(g_{n})\subset L^{r}(\R^{k})$ is a bounded sequence such that $g_{n}\rightarrow g$ a.e. in $\R^{k}$, then we have
\begin{align}\label{ggg}
|g_{n}-g|_{L^{r}(\R^{k})}^{r}= |g_{n}|_{L^{r}(\R^{k})}^{r} - |g|_{L^{r}(\R^{k})}^{r} +o_{n}(1).
\end{align}
Therefore
\begin{align*}
\int_{\R^{N}} V(\e x) |u_{n}-u|^{p}= \int_{\R^{N}} V(\e x) |u_{n}|^{p}- \int_{\R^{N}} V(\e x) |u|^{p}+ o_{n}(1),
\end{align*}
and taking
\begin{equation*}
g_{n}=\frac{u_{n}(x)-u_{n}(y)}{|x-y|^{\frac{N+sp}{p}}}, \quad  g= \frac{u(x)-u(y)}{|x-y|^{\frac{N+sp}{p}}}, \quad r=p \, \mbox{ and } \, k=2N
\end{equation*}
in \eqref{ggg}, we can see that
\begin{align*}
[u_{n}-u]^{p}_{s,p}= [u_{n}]^{p}_{s,p}- [u]^{p}_{s,p} + o_{n}(1).
\end{align*}
\end{proof}

\noindent
The next lemma is  a variant of the Brezis-Lieb Lemma \cite{BL} (see also \cite{Ack}) for the nonlocal term.
\begin{lem}\label{BLlem}
Let $(u_{n})\subset W^{s,p}(\R^{N})$ be such that $u_{n}\rightharpoonup u$ in $W^{s,p}(\R^{N})$. Set $v_{n}=u_{n}-u$.
Then we have
\begin{equation*}
\Sigma(u_{n})-\Sigma(v_{n})-\Sigma(u)=o_{n}(1),
\end{equation*}
and for any $\varphi\in \h$ such that $\|\varphi\|_{\e}\leq 1$ it holds
\begin{equation*}
\langle \Sigma'(u_{n})-\Sigma'(v_{n})-\Sigma'(u), \varphi\rangle=o_{n}(1),
\end{equation*}
where
$$
\Sigma(u):=\frac{1}{2}\int_{\R^{N}} K(u)(x) F(u(x))\, dx \quad \mbox{ and } \quad K(u)(x):=\int_{\R^{N}}\frac{F(u(y))}{|x-y|^{\mu}} dy.
$$
\end{lem}
\begin{proof}
We only show the validity of the first statement because the second one can be proved using similar arguments. For more details we refer the interested reader to \cite{Ack}.\\
Firstly, we show that 
\begin{equation}\label{BL1}
F(u_{n})-F(v_{n})\rightarrow F(u) \mbox{ in } L^{\frac{2N}{2N-\mu}}(\R^{N}).
\end{equation}
By the Mean Value Theorem, assumption $(f_1)$ and Young's inequality we can see that for any $\e>0$ there exists $C_{\e}>0$ such that
\begin{align*}
|F(u_{n})-F(v_{n})|^{\frac{2N}{2N-\mu}}&\leq \left|\int_{0}^{1} f(u_{n}-tu)u \, dt \right|^{\frac{2N}{2N-\mu}} \\
&\leq \left[|u| (|u_{n}|+|u|)^{q_{1}-1} +|u|(|u_{n}|+|u|)^{q_{2}-1} \right]^{\frac{2N}{2N-\mu}} \\
&\leq \e(|u_{n}|^{\frac{2Nq_{1}}{2N-\mu}}+|u_{n}|^{\frac{2Nq_{2}}{2N-\mu}})+C (|u|^{\frac{2Nq_{1}}{2N-\mu}}+|u|^{\frac{2Nq_{2}}{2N-\mu}})
\end{align*}
which together with
$$
|F(u)|^{\frac{2N}{2N-\mu}}\leq C (|u|^{\frac{2Nq_{1}}{2N-\mu}}+|u|^{\frac{2Nq_{2}}{2N-\mu}})
$$
implies that
$$
|F(u_{n})-F(v_{n})-F(u)|^{\frac{2N}{2N-\mu}}\leq \e(|u_{n}|^{\frac{2Nq_{1}}{2N-\mu}}+|u_{n}|^{\frac{2Nq_{2}}{2N-\mu}})+C (|u|^{\frac{2Nq_{1}}{2N-\mu}}+|u|^{\frac{2Nq_{2}}{2N-\mu}}).
$$
Let us define 
$$
G_{\e, n}=\max \left\{|F(u_{n})-F(v_{n})-F(u)|^{\frac{2N}{2N-\mu}}-\e(|u_{n}|^{\frac{2Nq_{1}}{2N-\mu}}+|u_{n}|^{\frac{2Nq_{2}}{2N-\mu}}), 0 \right\},
$$
and  we observe that $G_{\e, n}\rightarrow 0$ a.e. in $\R^{N}$ as $n\rightarrow \infty$ and 
$$
0\leq G_{\e, n}\leq C(|u|^{\frac{2Nq_{1}}{2N-\mu}}+|u|^{\frac{2Nq_{2}}{2N-\mu}})\in L^{1}(\R^{N}),
$$ 
because of $p<\frac{2Nq_{1}}{2N-\mu}<\p$ and $p<\frac{2Nq_{2}}{2N-\mu}<\p$ in view of $p<q_{1}\leq q_{2}<\frac{p(N-\mu)}{N-sp}$.\\
Hence, invoking the Dominated Convergence Theorem, we get $G_{\e, n}\rightarrow 0$ in $L^{1}(\R^{N})$ as $n\rightarrow \infty$. \\
On the other hand, from the definition of $G_{\e, n}$ it follows that 
$$
|F(u_{n})-F(v_{n})-F(u)|^{\frac{2N}{2N-\mu}}\leq \e(|u_{n}|^{\frac{2Nq_{1}}{2N-\mu}}+|u_{n}|^{\frac{2Nq_{2}}{2N-\mu}})+G_{\e, n},
$$
so, using the boundedness of $(u_{n})$ in $W^{s,p}(\R^{N})$, we can deduce that
$$
\limsup_{n\rightarrow \infty} \int_{\R^{N}} |F(u_{n})-F(v_{n})-F(u)|^{\frac{2N}{2N-\mu}}\leq C\e \quad \forall \e>0.
$$
This ends the proof of \eqref{BL1}.
Then, in view of Theorem \ref{HLS}, we have
\begin{equation}\label{BL2}
\frac{1}{|x|^{\mu}}*(F(u_{n})-F(v_{n}))\rightarrow \frac{1}{|x|^{\mu}}*F(u) \mbox{ in } L^{\frac{2N}{\mu}}(\R^{N}).
\end{equation}
Now, let us note that
\begin{align}\label{BL3}
&\int_{\R^{N}} \left(\frac{1}{|x|^{\mu}}*F(u_{n})\right)F(u_{n})-\left(\frac{1}{|x|^{\mu}}*F(v_{n})\right)F(v_{n}) \nonumber\\
&=\int_{\R^{N}} \left(\frac{1}{|x|^{\mu}}*(F(u_{n})-F(v_{n}))\right) (F(u_{n})-F(v_{n}))+2\int_{\R^{N}} \left(\frac{1}{|x|^{\mu}}*(F(u_{n})-F(v_{n}))\right) F(v_{n}).
\end{align}
Since $F(v_{n})\rightharpoonup 0$ in $L^{\frac{2N}{2N-\mu}}(\R^{N})$, we can use \eqref{BL1}, \eqref{BL2} and  \eqref{BL3} to obtain the thesis.
\end{proof}

\section{Variational framework}
In order to study \eqref{R}, we will look for critical points of the following Euler-Lagrange functional $\J_{\e}: \h\rightarrow \R$ defined by
$$
\J_{\e}(u)=\frac{1}{p}\|u\|^{p}_{\e}-\Sigma(u)
$$
where 
$$
\Sigma(u)=\frac{1}{2}\int_{\R^{N}} K(u)(x) F(u(x))\, dx \quad \mbox{ and } \quad K(u)(x)=\int_{\R^{N}}\frac{F(u(y))}{|x-y|^{\mu}} dy.
$$
In view of $(f_1)$, Theorem \ref{embedding} and Theorem \ref{HLS}, it is easy to check that $\J_{\e}$ is well-defined, $\J_{\e}\in C^{1}(\h, \R)$ and its differential is given by
\begin{align*}
\langle \J'_{\e}(u), v\rangle&=\iint_{\R^{2N}} \frac{|u(x)-u(y)|^{p-2}(u(x)-u(y))(v(x)-v(y))}{|x-y|^{N+sp}}\, dx dy+\int_{\R^{N}} V(\e x)|u|^{p-2}uv \,dx\\
&\quad -\int_{\R^{N}}\int_{\R^{N}} \frac{F(u(y))}{|x-y|^{\mu}}f(u(x)) v(x) \,dx dy
\end{align*}
for any $u, v\in \h$.
Let us introduce the Nehari manifold associated with (\ref{R}), namely
\begin{equation*}
\mathcal{N}_{\e}:= \{u\in \h \setminus \{0\} : \langle \J_{\e}'(u), u \rangle =0\}.
\end{equation*}

\noindent
Firstly, we show that $\J_{\e}$ verifies the assumptions of the Mountain Pass Theorem \cite{AR}. 
\begin{lem}\label{MPG}
$\J_{\e}$ verifies the following properties:
\begin{compactenum}[(i)]
\item there exist $\alpha, \rho>0$ such that $\J_{\e}(u)\geq \alpha$ for any $u\in \h$ such that $\|u\|_{\e}=\rho$;
\item there exists $e\in \h$ with $\|e\|_{\e}>\rho$ such that $\J_{\e}(e)<0$.
\end{compactenum}
\end{lem}
\begin{proof}
Using $(f_1)$ and applying Theorem \ref{HLS} we get
\begin{align}\label{a1}
\left|\int_{\R^{N}} K(u)F(u)\, dx\right|\leq C|F(u)|_{t} |F(u)|_{t} \leq C\left(\int_{\R^{N}} (|u|^{q_{1}}+|u|^{q_{2}}\, dx)^{t}\right)^{\frac{2}{t}},
\end{align}
where $\frac{1}{t}=\frac{1}{2}(2-\frac{\mu}{N})$. Since $p<q_{1}\leq q_{2}<\frac{p^{*}_{s}}{2}(2-\frac{\mu}{N})$, we can see that $t q_{1}, tq_{2}\in (p, p^{*}_{s})$, and from Theorem \ref{embedding}  we have
\begin{align}\label{a2}
\left(\int_{\R^{N}} (|u|^{q_{1}}+|u|^{q_{2}})^{t}\, dx\right)^{\frac{2}{t}}\leq C(\|u\|^{q_{1}}_{\e}+\|u\|^{q_{2}}_{\e})^{2}.
\end{align}
Putting together \eqref{a1} and \eqref{a2} we can deduce that
\begin{align*}
\left|\int_{\R^{N}} \left(\frac{1}{|x|^{\mu}}*F(u)\right)F(u)\, dx\right|\leq C(\|u\|^{q_{1}}_{\e}+\|u\|^{q_{2}}_{\e})^{2}\leq C(\|u\|^{2q_{1}}_{\e}+\|u\|^{2q_{2}}_{\e}).
\end{align*}
Therefore, we obtain
$$
\J(u)\geq \frac{1}{p}\|u\|^{p}_{\e}-C(\|u\|^{2q_{1}}_{\e}+\|u\|^{2q_{2}}_{\e}),
$$
and being $q_{2}\geq q_{1}>\frac{p}{2}$ we can see that $(i)$ holds.
Fix $u_{0}\in W^{s,p}(\R^{N})\setminus\{0\}$ such that $u_{0}\geq 0$, and we set
$$
h(t)=\Sigma\left(\frac{t u_{0}}{\|u_{0}\|_{\e}}\right) \mbox{ for } t>0.
$$
Using $(f_2)$, we deduce that
\begin{align}\label{a3}
h'(t)&=\Sigma'\left(\frac{t u_{0}}{\|u_{0}\|_{\e}}\right) \frac{u_{0}}{\|u_{0}\|_{\e}} \nonumber \\
&=\int_{\R^{N}} \left(\frac{1}{|x|^{\mu}}*F\left(\frac{t u_{0}}{\|u_{0}\|_{\e}}\right)  \right) f\left(\frac{t u_{0}}{\|u_{0}\|_{\e}}\right)\frac{u_{0}}{\|u_{0}\|_{\e}}\, dx\nonumber \\
&=\frac{1}{t}\int_{\R^{N}} \frac{1}{2} \left(\frac{1}{|x|^{\mu}}*F\left(\frac{t u_{0}}{\|u_{0}\|_{\e}}\right)  \right) 2  f\left(\frac{t u_{0}}{\|u_{0}\|_{\e}}\right)\frac{t u_{0}}{\|u_{0}\|_{\e}}\, dx \nonumber\\
&>\frac{\theta}{t} h(t).
\end{align}
Integrating \eqref{a3} on $[1, t\|u_{0}\|_{\e}]$ with $t>\frac{1}{\|u_{0}\|_{\e}}$, we find
$$
h(t\|u_{0}\|_{\e})\geq h(1)(t\|u_{0}\|_{\e})^{\theta}
$$
which implies that
$$
\Sigma(t u_{0})\geq \Sigma\left(\frac{u_{0}}{\|u_{0}\|_{\e}}\right) \|u_{0}\|^{\theta}_{\e}t^{\theta}.
$$
Consequently, we have
$$
\J_{\e}(t u_{0})= \frac{t^{p}}{p}\|u_{0}\|^{p}_{\e}-\Sigma(t u_{0})\leq C_{1} t^{p}-C_{2}t^{\theta} \mbox{ for } t>\frac{1}{\|u_{0}\|}_{\e}.
$$
Taking $e=t u_{0}$ with $t$ sufficiently large, we can see that $(ii)$ holds.
\end{proof}

\noindent
Now, we prove the following lemma related to the function $K(u)$ which will be very useful in the sequel.
\begin{lem}\label{lemK}
Assume that $(f_1)$-$(f_3)$ hold, $0<\mu<sp$ and $p<q_{1}\leq q_{2}<\frac{p(N-\mu)}{N-sp}$. Let $(u_{n})$ be a bounded sequence in $\h$. Then there exists $C_{0}>0$ such that
\begin{equation}\label{a6}
|K(u_{n})|_{\infty}\leq C_{0} \mbox{ for any } \e>0.
\end{equation}
\end{lem}
\begin{proof}
Let us note that $(f_1)$ yields
\begin{equation*}
|F(t)|\leq C(|t|^{q_{1}}+|t|^{q_{2}}) \mbox{ for all } t\in \R.
\end{equation*}
Then, we can see that
\begin{align}\label{a7}
|K(u_{n})(x)|&=\left| \int_{\R^{N}} \frac{F(u_{n})}{|x-y|^{\mu}} \,dy\right| \nonumber\\
&\leq \left| \int_{|x-y|\leq 1} \frac{F(u_{n})}{|x-y|^{\mu}} \,dy\right|+\left| \int_{|x-y|>1} \frac{F(u_{n})}{|x-y|^{\mu}} \,dy\right| \nonumber\\
&\leq C \int_{|x-y|\leq 1} \frac{|u_{n}(y)|^{q_{1}}+|u_{n}(y)|^{q_{2}}}{|x-y|^{\mu}}\, dy+C \int_{\R^{N}} (|u_{n}|^{q_{1}}+|u_{n}|^{q_{2}})\, dy \nonumber\\
&\leq C \int_{|x-y|\leq 1} \frac{|u_{n}(y)|^{q_{1}}+|u_{n}(y)|^{q_{2}}}{|x-y|^{\mu}}\, dy+C
\end{align}
where in the last line we used Theorem \ref{embedding} and $\|u_{n}\|_{\e}\leq K$.
Now, we take 
$$
t\in \left(\frac{N}{N-\mu}, \frac{Np}{(N-sp)q_{1}}\right] \mbox{ and } r\in \left(\frac{N}{N-\mu}, \frac{Np}{(N-sp)q_{2}}\right].
$$
Applying the H\"older inequality and using Theorem \ref{embedding} and $\|u_{n}\|_{\e}\leq K$, we can see that
\begin{align}\label{a8}
\int_{|x-y|\leq 1} \frac{|u_{n}(y)|^{q_{1}}}{|x-y|^{\mu}}\, dy&\leq \left(\int_{|x-y|\leq 1} |u_{n}|^{tq_{1}}\, dy  \right)^{\frac{1}{t}} \left(\int_{|x-y|\leq 1} \frac{1}{|x-y|^{\frac{t\mu}{t-1}}}\, dy \right)^{\frac{t-1}{t}}\nonumber \\
&\leq C \left(\int_{0}^{1} \rho^{N-1-\frac{t \mu}{t-1}}\, d\rho  \right)^{\frac{t-1}{t}}<\infty.
\end{align}
because of $N-1-\frac{t \mu}{t-1}>-1$.
Similarly, we get
\begin{align}\label{a9}
\int_{|x-y|\leq 1} \frac{|u_{n}(y)|^{q_{2}}}{|x-y|^{\mu}}\, dy&\leq \left(\int_{|x-y|\leq 1} |u_{n}|^{rq_{2}}\, dy  \right)^{\frac{1}{r}} \left(\int_{|x-y|\leq 1} \frac{1}{|x-y|^{\frac{r\mu}{r-1}}}\, dy \right)^{\frac{r-1}{r}}\nonumber \\
&\leq C \left(\int_{0}^{1} \rho^{N-1-\frac{r \mu}{r-1}}\, d\rho  \right)^{\frac{r-1}{r}}<\infty
\end{align}
in view of $N-1-\frac{r \mu}{r-1}>-1$.
Putting together \eqref{a8} and \eqref{a9} we can see that
$$
\int_{|x-y|\leq 1} \frac{|u_{n}(y)|^{q_{1}}+|u_{n}(y)|^{q_{2}}}{|x-y|^{\mu}}\, dy\leq C \mbox{ for all } x\in \R^{N}
$$
which in view of \eqref{a7} yields \eqref{a6}.
\end{proof}

Since $f$ is only continuous, the next results are very important because they allow us to overcome the non-differentiability of $\mathcal{N}_{\e}$. We begin proving some properties for the functional $\J_{\e}$.
\begin{lem}\label{SW1}
Under assumptions $(V)$ and $(f_1)$-$(f_3)$ we have for any $\e>0$:
\begin{compactenum}[$(i)$]
\item $\J'_{\e}$ maps bounded sets of $\h$ into bounded sets of $\h$.
\item $\J'_{\e}$ is weakly sequentially continuous in $\h$.
\item $\J_{\e}(t_{n}u_{n})\rightarrow -\infty$ as $t_{n}\rightarrow \infty$, where $u_{n}\in K$ and $K\subset \h\setminus\{0\}$ is a compact subset.
\end{compactenum}
\end{lem}
\begin{proof}
$(i)$ Let $(u_{n})$ be a bounded sequence in $\h$ and $v \in \h$. Then, from $(f_{1})$ and Lemma \ref{lemK}  we deduce that
\begin{align*}
\langle \J'_{\e}(u_{n}), v \rangle &\leq C \|u_{n}\|_{\e}^{p-1} \|v\|_{\e} + C \|u_{n}\|_{\e}^{q_{1}-1} \|v\|_{\e}+C\|u_{n}\|_{\e}^{q_{2}-1} \|v\|_{\e} \leq C.
\end{align*}

$(ii)$ Assume that $u_{n}\rightharpoonup u$ in $\h$ and take $v\in C^{\infty}_{c}(\R^{N})$.
Then, we know that
\begin{align*}
\langle \J'_{\e}(u_{n}), v\rangle&=\iint_{\R^{2N}} \frac{|u_{n}(x)-u_{n}(y)|^{p-2}(u_{n}(x)-u_{n}(y))(v(x)-v(y))}{|x-y|^{N+sp}}\, dx dy+\int_{\R^{N}} V(\e x)|u_{n}|^{p-2}u_{n}v \,dx\\
&\quad -\int_{\R^{N}}\int_{\R^{N}} \frac{F(u_{n}(y))}{|x-y|^{\mu}}f(u_{n}(x)) v(x) \, dx dy.
\end{align*}
The weak convergence gives that 
\begin{align*}
\iint_{\R^{2N}} &\frac{|u_{n}(x)-u_{n}(y)|^{p-2}(u_{n}(x)-u_{n}(y))(v(x)-v(y))}{|x-y|^{N+sp}}\, dx dy \\
&\quad \rightarrow \iint_{\R^{2N}} \frac{|u(x)-u(y)|^{p-2}(u(x)-u(y))(v(x)-v(y))}{|x-y|^{N+sp}}\, dx dy
\end{align*}
and
$$
\int_{\R^{N}} V(\e x)|u_{n}|^{p-2}u_{n}v \,dx\rightarrow \int_{\R^{N}} V(\e x)|u|^{p-2}u v \,dx.
$$
Now, the growth conditions on $f$ and the boundedness of $(u_{n})$ in $\h$ imply that $F(u_{n})$ is bounded in $L^{\frac{2N}{2N-\mu}}(\R^{N})$. Moreover, $u_{n}\rightarrow u$ a.e. in $\R^{N}$ and the continuity of $F$ gives that $F(u_{n})\rightarrow F(u)$ a.e. in $\R^{N}$.
Therefore, $F(u_{n})\rightharpoonup F(u)$ in $L^{\frac{2N}{2N-\mu}}(\R^{N})$. Using Theorem \ref{HLS}, we know that the convolution term
$$
\frac{1}{|x|^{\mu}}*w(x)\in L^{\frac{2N}{\mu}}(\R^{N}) \quad \forall w\in L^{\frac{2N}{2N-\mu}}(\R^{N}),
$$
and it is a linear bounded operator from $L^{\frac{2N}{2N-\mu}}(\R^{N})$ to $L^{\frac{2N}{\mu}}(\R^{N})$. Accordingly,
$$
\frac{1}{|x|^{\mu}}*F(u_{n})\rightharpoonup \frac{1}{|x|^{\mu}}*F(u) \mbox{ in } L^{\frac{2N}{\mu}}(\R^{N}).
$$
Since $f$ has subcritical growth, we know that $f(u_{n})\rightarrow f(u)$ in $L^{r}_{loc}(\R^{N})$ for all $r\in [1, \frac{\p}{q_{2}-1})$.
Therefore, we get
$$
\int_{\R^{N}} \left( \frac{1}{|x|^{\mu}}*F(u_{n}) \right)f(u_{n}) v\rightarrow \int_{\R^{N}} \left( \frac{1}{|x|^{\mu}}*F(u) \right)f(u)  v.
$$
In conclusion, $\langle \J_{\e}'(u_{n}), v\rangle\rightarrow \langle \J_{\e}'(u), v\rangle$ for all $v\in C^{\infty}_{c}(\R^{N})$, and using the density of $C^{\infty}_{c}(\R^{N})$ in $\h$ we get the thesis.

$(iii)$ Without loss of generality, we may assume that $\|u\|_{\e}=1$ for each $u\in K$. For $u_{n}\in K$, after passing to a subsequence, we obtain that $u_{n}\rightarrow u\in \mathbb{S}_{\e}$. Then, using $(f_{2})$ and Fatou's Lemma we can see that
\begin{align*}
\J_{\e}(t_{n}u_{n}) \leq C t_{n}^{p}- Ct_{n}^{\theta}  \rightarrow -\infty \, \mbox{ as } n\rightarrow \infty
\end{align*}
where we used that $\theta>p$.
\end{proof}

\begin{lem}\label{SW2}
Under the assumptions of Lemma \ref{SW1}, for $\e>0$ we have:
\begin{compactenum}[$(i)$]
\item for all $u\in \mathbb{S}_{\e}$, there exists a unique $t_{u}>0$ such that $t_{u}u\in \mathcal{N}_{\e}$. Moreover, $m_{\e}(u)=t_{u}u$ is the unique maximum of $\J_{\e}$ on $\h$, where $\mathbb{S}_{\e}=\{u\in \h: \|u\|_{\e}=1\}$.
\item The set $\mathcal{N}_{\e}$ is bounded away from $0$. Furthermore $\mathcal{N}_{\e}$ is closed in $\h$.
\item There exists $\alpha>0$ such that $t_{u}\geq \alpha$ for each $u\in \mathbb{S}_{\e}$ and, for each compact subset $W\subset \mathbb{S}_{\e}$, there exists $C_{W}>0$ such that $t_{u}\leq C_{W}$ for all $u\in W$.
\item For each $u\in \mathcal{N}_{\e}$, $m_{\e}^{-1}(u)=\frac{u}{\|u\|_{\e}}\in \mathcal{N}_{\e}$. In particular, $\mathcal{N}_{\e}$ is a regular manifold diffeomorphic to the sphere in $\h$.
\item $c_{\e}=\inf_{\mathcal{N}_{\e}} \J_{\e}\geq \rho>0$ and $\J_{\e}$ is bounded below on $\mathcal{N}_{\e}$, where $\rho$ is independent of $\e$.
\end{compactenum}
\end{lem}

\begin{proof}
$(i)$ For each $u\in \mathbb{S}_{\e}$ and $t>0$, we define $h(t)=\J_{\e}(tu)$. From the proof of the Lemma \ref{MPG} we know that $h(0)=0$, $h(t)<0$ for $t$ large and $h(t)>0$ for $t$ small. Therefore, $\max_{t\geq 0} h(t)$ is achieved at some $t=t_{u}>0$ satisfying $h'(t_{u})=0$ and $t_{u}u\in \mathcal{N}_{\e}$. Now, we note that $tu\in \mathcal{N}_{\e}$ if and only if
\begin{align}\label{STAR}
\|u\|^{p}_{\e}=\int_{\R^{N}}\int_{\R^{N}} \frac{F(t u(y))}{t^{\frac{p}{2}}|x-y|^{\mu}} \frac{f(tu(x))}{t^{\frac{p}{2}-1}} u(x)\, dx dy.
\end{align}
Using $(f_{2})$ and $(f_{3})$, we can see that the functions 
$$
t\mapsto \frac{F(t)}{t^{\frac{p}{2}}}\quad \mbox{ and }\quad t\mapsto \frac{f(t)}{t^{\frac{p}{2}-1}}
$$
are increasing for $t>0$, so the right hand side in \eqref{STAR} is an increasing function of $t$. Then, it is easy to verify the uniqueness of a such $t_{u}$.

\noindent
$(ii)$ Using Theorem \ref{HLS} and $(f_1)$, we can see that for all $u\in \mathcal{N}_{\e}$
$$
\|u\|^{p}_{\e}\leq C(\|u\|_{\e}^{2q_{1}}+\|u\|^{2q_{2}}_{\e})
$$
so there exists $\kappa>0$ such that 
\begin{equation}\label{consnehari}
\|u\|_{\e}\geq \kappa.
\end{equation}

Now we prove that the set $\mathcal{N}_{\e}$ is closed in $\h$. Let $(u_{n})\subset \mathcal{N}_{\e}$ such that $u_{n}\rightarrow u$ in $\h$. In view of Lemma \ref{SW1} we know that $\J'_{\e}(u_{n})$ is bounded, so we can deduce that
\begin{align*}
\langle \J'_{\e}(u_{n}), u_{n} \rangle - \langle \J'_{\e}(u), u \rangle= \langle \J'_{\e}(u_{n})- \J'_{\e}(u), u \rangle + \langle \J'_{\e}(u_{n}), u_{n}-u \rangle \rightarrow 0
\end{align*}
that is $\langle \J'_{\e}(u), u\rangle=0$. This combined with $\|u\|_{\e}\geq \kappa$ implies that
$$
\|u\|_{\e}= \lim_{n\rightarrow \infty} \|u_{n}\|_{\e} \geq \kappa >0,
$$
that is $u\in \mathcal{N}_{\e}$.

\noindent
$(iii)$ For each $u\in \mathbb{S}_{\e}$ there exists $t_{u}>0$ such that $t_{u}u\in \mathcal{N}_{\e}$. From the proof of $(ii)$, we can see that
$$
t_{u}= \|t_{u}u\|_{\e} \geq \kappa.
$$
Now we prove that $t_{u}\leq C_{W}$ for all $u\in W\subset \mathbb{S}_{\e}$. Assume by contradiction that there exists $(u_{n})\subset W\subset \mathbb{S}_{\e}$ such that $t_{u_{n}}\rightarrow \infty$. Since $W$ is compact, there is $u\in W$ such that $u_{n}\rightarrow u$ in $\h$ and $u_{n}\rightarrow u$ a.e. in $\R^{N}$. Using Lemma \ref{SW1}-$(iii)$, we can infer that $\J_{\e}(t_{u_{n}}u_{n})\rightarrow -\infty$ as $n\rightarrow \infty$, which gives a contradiction because $(f_{2})$ implies that
$$
\J_{\e}(u)|_{\mathcal{N}_{\e}}= \int_{\R^{N}} K(u)\left[\frac{1}{p}f(u)u- \frac{1}{2} F(u)\right]\, dx \geq 0.
$$

\noindent
$(iv)$ Let us define the maps $\hat{m}_{\e}: \h\setminus \{0\} \rightarrow \mathcal{N}_{\e}$ and $m_{\e}: \mathbb{S}_{\e}\rightarrow \mathcal{N}_{\e}$ by 
\begin{align}\label{me}
\hat{m}_{\e}(u)= t_{u}u \quad \mbox{ and } \quad m_{\e}= \hat{m}_{\e}|_{\mathbb{S}_{\e}}.
\end{align}
In the light of $(i)$-$(iii)$, we can apply Proposition 8 in \cite{SW} to deduce that $m_{\e}$ is a homeomorphism between $\mathbb{S}_{\e}$ and $\mathcal{N}_{\e}$ and the inverse of $m_{\e}$ is given by $m_{\e}^{-1}(u)=\frac{u}{\|u\|_{\e}}$. Therefore $\mathcal{N}_{\e}$ is a regular manifold diffeomorphic to $\mathbb{S}_{\e}$.

\noindent
$(v)$ For $\e>0$, $t>0$ and $u\in \h \setminus \{0\}$, we can argue as in Lemma \ref{MPG} to see that
\begin{align*}
\J_{\e}(tu)\geq \frac{t^{p}}{p}\|u\|_{\e}^{p}-  C(t^{2q_{1}} \|u\|_{\e}^{2q_{1}}+t^{q_{2}}\|u\|_{\e}^{2q_{2}}).
\end{align*}
Hence, we can find $\rho>0$ such that $\J_{\e}(tu)\geq \rho>0$ for $t>0$ small enough. On the other hand, using $(i)$-$(iii)$, we know (see \cite{SW}) that
\begin{align}\label{Nguyen1}
c_{\e}= \inf_{u\in \mathcal{N}_{\e}} \J_{\e}(u)= \inf_{u\in \h\setminus \{0\}} \max_{t\geq 0} \J_{\e}(tu) = \inf_{u\in \mathbb{S}_{\e}} \max_{t\geq 0} \J_{\e}(tu)
\end{align}
which yields $c_{\e}\geq \rho$ and $\J_{\e}|_{\mathcal{N}_{\e}}\geq \rho$.
\end{proof}

Now we introduce the functionals $\hat{\Psi}_{\e}: \h\setminus\{0\} \rightarrow \R$ and $\Psi_{\e}: \mathbb{S}_{\e}\rightarrow \R$ defined by
\begin{align*}
\hat{\Psi}_{\e}= \J_{\e}(\hat{m}_{\e}(u)) \quad \mbox{ and } \quad \Psi_{\e}= \hat{\Psi}_{\e}|_{\mathbb{S}_{\e}},
\end{align*}
where $\hat{m}_{\e}(u)= t_{u}u$ is given in \eqref{me}.
As in \cite{SW} we have the following result:

\begin{lem}\label{SW3}
Under the assumptions of Lemma \ref{SW1}, we have that for $\e>0$:
\begin{compactenum}[$(i)$]
\item $\Psi_{\e}\in C^{1}(\mathbb{S}_{\e}, \R)$, and
\begin{equation*}
\Psi'_{\e}(w)v= \|m_{\e}(w)\|_{\e} \J'_{\e}(m_{\e}(w))v \quad \mbox{ for } v\in T_{w}(\mathbb{S}_{\e}).
\end{equation*}
\item $(w_{n})$ is a Palais-Smale sequence for $\Psi_{\e}$ if and only if $(m_{\e}(w_{n}))$ is a Palais-Smale sequence for $\J_{\e}$. If $(u_{n})\subset \mathcal{N}_{\e}$ is a bounded Palais-Smale sequence for $\J_{\e}$, then $(m_{\e}^{-1}(u_{n}))$ is a Palais-Smale sequence for $\Psi_{\e}$.
\item $u\in \mathbb{S}_{\e}$ is a critical point of $\Psi_{\e}$ if and only if $m_{\e}(u)$ is a critical point of $\J_{\e}$. Moreover the corresponding critical values coincide and
\begin{equation*}
\inf_{\mathbb{S}_{\e}} \Psi_{\e}=\inf_{\mathcal{N}_{\e}} \J_{\e}=c_{\e}.
\end{equation*}
\end{compactenum}
\end{lem}

\noindent
Using a variant of the Mountain Pass Theorem without Palais-Smale condition \cite{W}, we know that there exists a Palais-Smale sequence $(u_{n})\subset \h$ at the level $c_{\e}$ such that
$$
\J_{\e}(u_{n})\rightarrow c_{\e} \mbox{ and } \J'_{\e}(u_{n})\rightarrow 0.
$$
\begin{lem}\label{lemB}
Let $c\in \R$ and $(u_{n})$ be a Palais-Smale sequence of $\J_{\e}$ at level $c$. Then $(u_{n})$ is bounded in $\h$.
\end{lem}
\begin{proof}
Using assumption $(f_{2})$ (which implies that $K(u_{n})\geq 0$) we have
\begin{align*}
c+o_{n}(1)\|u_{n}\|_{\e} &= \J_{\e}(u_{n})- \frac{1}{\theta} \langle \J'_{\e}(u_{n}), u_{n} \rangle \\
&= \left( \frac{1}{p}- \frac{1}{\theta}\right) \|u_{n}\|_{\e}^{p} + \frac{1}{\theta} \int_{\R^{N}} K(u_{n})\left( f(u_{n})u_{n}- \frac{\theta}{2} F(u_{n})\right) \, dx \\
&\geq \left( \frac{1}{p}- \frac{1}{\theta}\right) \|u_{n}\|_{\e}^{p},
\end{align*}
and being $\theta>p$ we get the thesis.
\end{proof}

\section{The limit problem}

In this section we deal with the autonomous problem associated with \eqref{R}, that is
\begin{equation}\tag{$P_{\mu}$}
\left\{
\begin{array}{ll}
(-\Delta)^{s}_{p} u + \mu |u|^{p-2}u =  \left(\frac{1}{|x|^{\mu}}*F(u)\right)f(u)  \mbox{ in } \R^{N},  \\
u\in W^{s,p}(\R^{N}), \quad u>0 \mbox{ in } \R^{N},
\end{array}
\right.
\end{equation}
where $\mu>0$.
The corresponding functional is given by
\begin{equation*}
\I_{\mu}(u)=\frac{1}{p} \|u\|_{\mu}^{p} - \Sigma(u)
\end{equation*}
which is well defined on the space $\X_{\mu}=W^{s, p}(\R^{N})$ endowed with the norm
\begin{equation*}
\|u\|_{\mu}^{p}:=  [u]_{s,p}^{p}+\mu |u|^{p}_{p}.
\end{equation*}
Hence, $\I_{\mu}\in C^{1}(\X_{\mu}, \R)$ and its differential $\I'_{\mu}$ is given by
\begin{align*}
\langle \I'_{\mu}(u), \varphi \rangle &= \iint_{\R^{2N}} \frac{|u(x)-u(y)|^{p-2}(u(x)- u(y))}{|x-y|^{N+sp}} (\varphi(x)- \varphi(y)) \,dxdy \\
&+ \mu \int_{\R^{N}} |u|^{p-2} u\, \varphi \, dx -\int_{\R^{N}}\int_{\R^{N}} \frac{F(u(y))}{|x-y|^{\mu}}f(u(x)) v(x) \,dx dy
\end{align*}
for any $u, \varphi \in \X_{\mu}$.
Let us define the Nehari manifold associated with $\I_{\mu}$, that is
\begin{equation*}
\mathcal{M}_{\mu}= \left\{u\in \X_{\mu}\setminus \{0\} : \langle \I'_{\mu}(u), u\rangle =0  \right\}.
\end{equation*}
Arguing as in Section $3$ we can prove the following lemma.

\begin{lem}\label{SW2A}
Under the assumptions of Lemma \ref{SW1}, for $\mu>0$ we have:
\begin{compactenum}[$(i)$]
\item for all $u\in \mathbb{S}_{\mu}$, there exists a unique $t_{u}>0$ such that $t_{u}u\in \M_{\mu}$. Moreover, $m_{\mu}(u)=t_{u}u$ is the unique maximum of $\I_{\mu}$ on $\h$, where $\mathbb{S}_{\mu}=\{u\in \X_{\mu}: \|u\|_{\mu}=1\}$.
\item The set $\M_{\mu}$ is bounded away from $0$. Furthermore $\X_{\mu}$ is closed in $\X_{\mu}$.
\item There exists $\alpha>0$ such that $t_{u}\geq \alpha$ for each $u\in \mathbb{S}_{\mu}$ and, for each compact subset $W\subset \mathbb{S}_{\mu}$, there exists $C_{W}>0$ such that $t_{u}\leq C_{W}$ for all $u\in W$.
\item $\mathcal{M}_{\mu}$ is a regular manifold diffeomorphic to the sphere in $\X_{\mu}$.
\item $d_{\mu}=\inf_{\mathcal{M}_{\mu}} \I_{\mu}>0$ and $\I_{\mu}$ is bounded below on $\M_{\mu}$ by some positive constant.
\item $\I_{\mu}$ is coercive on $\M_{\mu}$.
\end{compactenum}
\end{lem}

Now we define the following functionals $\hat{\Psi}_{\mu}: \X_{\mu}\setminus\{0\} \rightarrow \R$ and $\Psi_{\mu}: \mathbb{S}_{\mu}\rightarrow \R$ by 
\begin{align*}
\hat{\Psi}_{\mu}= \I_{\mu}(\hat{m}_{\mu}(u)) \quad \mbox{ and } \quad \Psi_{\mu}= \hat{\Psi}_{\mu}|_{\mathbb{S}_{\mu}}.
\end{align*}

\begin{lem}\label{SW3A}
Under the assumptions of Lemma \ref{SW1}, we have that for $\mu>0$:
\begin{compactenum}[$(i)$]
\item $\Psi_{\mu}\in C^{1}(\mathbb{S}_{\mu}, \R)$, and
\begin{equation*}
\Psi'_{\mu}(w)v= \|m_{\mu}(w)\|_{\mu} \I'_{\mu}(m_{\mu}(w))v \quad \mbox{ for } v\in T_{w}(\mathbb{S}_{\mu}).
\end{equation*}
\item $(w_{n})$ is a Palais-Smale sequence for $\Psi_{\mu}$ if and only if $(m_{\mu}(w_{n}))$ is a Palais-Smale sequence for $\I_{\mu}$. If $(u_{n})\subset \M_{\mu}$ is a bounded Palais-Smale sequence for $\I_{\mu}$, then $(m_{\mu}^{-1}(u_{n}))$ is a Palais-Smale sequence for $\Psi_{\mu}$.
\item $u\in \mathbb{S}_{\mu}$ is a critical point of $\Psi_{\mu}$ if and only if $m_{\mu}(u)$ is a critical point of $\I_{\mu}$. Moreover the corresponding critical values coincide and
\begin{equation*}
\inf_{\mathbb{S}_{\mu}} \Psi_{\mu}=\inf_{\M_{\mu}} \I_{\mu}=d_{\mu}.
\end{equation*}
\end{compactenum}
\end{lem}

\begin{remark}
As in \eqref{Nguyen1}, from $(i)$-$(iii)$ of Lemma \ref{SW2A}, we can see that $d_{\mu}$ admits the following minimax characterization
\begin{align}\label{Nguyen2}
d_{\mu}= \inf_{u\in \M_{\mu}} \I_{\mu}(u)= \inf_{u\in \X_{\mu}\setminus \{0\}} \max_{t\geq 0} \I_{\mu}(tu) = \inf_{u\in \mathbb{S}_{\mu}} \max_{t\geq 0} \I_{\mu}(tu).
\end{align}
\end{remark}

\begin{lem}\label{lem2.2a}
Let $(u_{n})\subset \M_{\mu}$ be a minimizing sequence for $\I_{\mu}$. Then, $(u_{n})$ is bounded and there exist a sequence $(y_{n})\subset \R^{N}$ and constants $R, \beta>0$ such that
\begin{equation*}
\liminf_{n\rightarrow \infty} \int_{\B_{R}(y_{n})} |u_{n}|^{p} dx \geq \beta >0.
\end{equation*}
\end{lem}

\begin{proof}
Arguing as in the proof of Lemma \ref{lemB}, we can see that $(u_{n})$ is bounded in $\X_{\mu}$. 
Now, assume by contradiction that for any $R>0$ it holds
\begin{equation*}
\lim_{n\rightarrow \infty} \sup_{y\in \R^{N}} \int_{\B_{R}(y)} |u_{n}|^{p} dx=0.
\end{equation*}
Since $(u_{n})$ is bounded in $\X_{\mu}$, we can apply Lemma \ref{Lions} to see that
\begin{equation}\label{tv4N}
u_{n}\rightarrow 0 \, \mbox{ in } L^{t}(\R^{N}) \mbox{ for any } t\in (p, \p).
\end{equation}
Since $\langle \I'_{\mu}(u_{n}), u_{n} \rangle =0$ we get
$$
\|u_{n}\|^{p}_{\mu}=\int_{\R^{N}} K(u_{n}) f(u_{n})u_{n}dx.
$$
Taking into account $(f_1)$, Lemma \ref{lemK}, \eqref{tv4N} and the fact that $(u_{n})$ is bounded in $\X_{\mu}$, we have
\begin{align*}
0\leq \|u_{n}\|_{\mu}^{p}\leq C(|u_{n}|_{q_{1}}^{q_{1}}+|u_{n}|_{q_{2}}^{q_{2}})\rightarrow 0
\end{align*}
from which we deduce that $u_{n}\rightarrow 0$ in $\X_{\mu}$.
\end{proof}

Let us conclude this section proving the following existence result for $(P_{\mu})$.
\begin{lem}\label{lem4.3}
For all $\mu>0$, problem $(P_{\mu})$ has at least one positive ground state solution.
\end{lem}

\begin{proof}
From $(v)$ of Lemma \ref{SW2A}, we know that $d_{\mu}>0$ for each $\mu>0$. Moreover, if $u\in \M_{\mu}$ verifies $\I_{\mu}(u)=d_{\mu}$, then $m^{-1}_{\mu}(u)$ is a minimizer of $\Psi_{\mu}$ and it is a critical point of $\Psi_{\mu}$. In view of Lemma \ref{SW3A}, we can see that $u$ is a critical point of $\I_{\mu}$. Now we show that there exists a minimizer of $\I_{\mu}|_{\M_{\mu}}$. Applying Ekeland's variational principle there exists a sequence $(\nu_{n})\subset \mathbb{S}_{\mu}$ such that $\Psi_{\mu}(\nu_{n})\rightarrow d_{\mu}$ and $\Psi'_{\mu}(\nu_{n})\rightarrow 0$ as $n\rightarrow \infty$. Let $u_{n}=m_{\mu}(\nu_{n}) \in \M_{\mu}$. Then, thanks to Lemma \ref{SW3A}, $\I_{\mu}(u_{n})\rightarrow c_{\mu}$ and $\I'_{\mu}(u_{n})\rightarrow 0$ as $n\rightarrow \infty$.
Arguing as in Lemma \ref{lemB}, $(u_{n})$ is bounded in $\X_{\mu}$ and $u_{n}\rightharpoonup u$ in $\X_{\mu}$.
From Lemma \ref{lem2.2a}, we can find $(y_{n})\subset \R^{N}$ and $R, \beta>0$ such that
$$
\liminf_{n\rightarrow \infty} \int_{\B_{R}(y_{n})} |u_{n}|^{p} dx \geq \beta >0.
$$ 
Set $v_{n}(x)=u_{n}(x+y_{n})$. Then $\int_{\B_{R}(0)} |v_{n}|^{p} dx \geq \frac{\beta}{2}$. Since $\I_{\mu}$ and $\I'_{\mu}$ are invariant by translation, it holds that $\I_{\mu}(v_{n})\rightarrow d_{\mu}$ and $\I'_{\mu}(v_{n})\rightarrow 0$. Observing that $(v_{n})$ is bounded in $\X_{\mu}$, we may assume that $v_{n}\rightharpoonup v$ in $\X_{\mu}$, for some $v\neq 0$. Arguing as in $(ii)$ of Lemma  \ref{SW1}, we can deduce that $\I'_{\mu}(v)=0$. Since $v\neq 0$, we can deduce that $v\in \mathcal{M}_{\mu}$. Hence, $\I_{\mu}(v)\geq c_{\mu}$ and using Fatou's Lemma we can conclude that $\I_{\mu}(v)=d_{\mu}$. Now, recalling that $f(t)=0$ for $t\leq 0$ and $|x-y|^{q-2}(x-y)(x^{-}-y^{-})\geq |x^{-}-y^{-}|^{q}$ for all $q\geq 1$, we can deduce that $\langle \I'_{\mu}(v), v^{-}\rangle=0$ implies that $v\geq 0$ in $\R^{N}$. Arguing as in Lemma \ref{lemMoser}, we can obtain that $v\in L^{\infty}(\R^{N})$ and applying Corollary $5.5$ in \cite{IMS} we have $v\in C^{0, \alpha}(\R^{N})$. Using the maximum principle in \cite{DPQ} we can conclude that $u>0$ in $\R^{N}$.
\end{proof}

\section{Existence of a ground state solution}

\noindent
In this section we focus on the existence of a solution to \eqref{R} provided that $\e$ is sufficiently small. 
Firstly, we can note that arguing as in Lemma \ref{lem2.2a} we have the following result:
\begin{lem}\label{lem2.2}
Let $(u_{n})\subset \N_{\e}$ be a sequence for $\J_{\e}$ such that $\J_{\e}(u_{n})\rightarrow c$ and $u_{n}\rightharpoonup 0$ in $\h$. Then, one of the following alternatives occurs
\begin{compactenum}[$(a)$]
\item $u_{n}\rightarrow 0$ in $\h$; 
\item there are a sequence $(y_{n})\subset \R^{N}$ and constants $R, \beta>0$ such that 
\begin{equation*}
\liminf_{n\rightarrow \infty} \int_{\B_{R}(y_{n})} |u_{n}|^{p} dx \geq \beta >0. 
\end{equation*}
\end{compactenum}
\end{lem}

\begin{proof}
Assume that $(b)$ does not hold true. Then, for any $R>0$ it holds
\begin{equation*}
\lim_{n\rightarrow \infty} \sup_{y\in \R^{N}} \int_{\B_{R}(y)} |u_{n}|^{p} dx=0. 
\end{equation*}
Since $(u_{n})$ is bounded in $\h$, from Lemma \ref{Lions} it follows that 
\begin{equation}\label{tv4}
u_{n}\rightarrow 0 \mbox{ in } L^{r}(\R^{N}) \mbox{ for any } r\in (p, \p). 
\end{equation}
Then we can proceed as in Lemma \ref{lem2.2a} to get the thesis.
\end{proof}

\noindent
In order to get a compactness result for $\J_{\e}$, we need to prove the following auxiliary lemma.
\begin{lem}\label{lem2.3}
Assume that $V_{\infty}<\infty$ and let $(v_{n})	\subset \mathcal{N}_{\e}$ be a sequence such that $\J_{\e}(v_{n})\rightarrow d$ with $v_{n}\rightarrow 0$ in $\h$. If $v_{n}\nrightarrow 0$ in $\h$, then $d\geq d_{V_{\infty}}$.
\end{lem}

\begin{proof}
Let $(t_{n})\subset (0, +\infty)$ be such that $(t_{n}v_{n})\subset \mathcal{M}_{V_{\infty}}$. \\
\textsc{Claim 1}: We aim to prove that 
\begin{equation*}
\limsup_{n\rightarrow \infty} t_{n} \leq 1. 
\end{equation*}
Assume by contradiction that there exist $\delta>0$ and a subsequence, still denoted by $(t_{n})$, such that 
\begin{equation}\label{tv6}
t_{n}\geq 1+ \delta \quad \forall n\in \mathbb{N}. 
\end{equation}
Since $\langle \J'_{\e}(v_{n}), v_{n} \rangle =0$, we have
\begin{equation}\label{tv7}
[v_{n}]^{p}_{s,p} + \int_{\R^{N}} V(\e x) |v_{n}|^{p} dx = \int_{\R^{N}}\int_{\R^{N}} \frac{F(v_{n}(y)) f(v_{n}(x)) v_{n}(x)}{|x-y|^{\mu}}. 
\end{equation}
On the other hand, $t_{n}v_{n} \in \mathcal{M}_{V_{\infty}}$, so we get 
\begin{equation}\label{tv8}
t_{n}^{p} [v_{n}]^{p}_{s,p}+ t_{n}^{p} \int_{\R^{N}} V_{\infty} |v_{n}|^{p} dx = \int_{\R^{N}}\int_{\R^{N}} \frac{F(t_{n}v_{n}(y)) f(t_{n}v_{n}(x)) t_{n}v_{n}(x)}{|x-y|^{\mu}} . 
\end{equation}
Putting together \eqref{tv7} and \eqref{tv8} we obtain 
\begin{equation*}
\int_{\R^{N}} \left( V_{\infty} - V(\e x)\right) |v_{n}|^{p} dx=\int_{\R^{N}}\int_{\R^{N}} \left[\frac{F(t_{n}v_{n}(y)) f(t_{n}v_{n}(x)) v_{n}(x)}{t_{n}^{p-1}|x-y|^{\mu}}-  \frac{F(v_{n}(y)) f(v_{n}(x)) v_{n}(x)}{|x-y|^{\mu}} \right].
\end{equation*}
By assumption $(V)$ we can see that, given $\zeta>0$ there exists $R=R(\zeta)>0$ such that 
\begin{equation}\label{tv9}
V(\e x) \geq V_{\infty} - \zeta \quad \mbox{ for any } |x|\geq R. 
\end{equation}
Now, taking into account the fact that $v_{n}\rightarrow 0$ in $L^{p}(\B_{R}(0))$ and the boundedness of $(v_{n})$ in $\h$, we can infer that
\begin{align*}
\int_{\R^{N}} \left( V_{\infty} - V(\e x)\right) |v_{n}|^{p} dx&= \int_{\B_{R}(0)} \left( V_{\infty} - V(\e x)\right) |v_{n}|^{p} dx+ \int_{\B_{R}^{c}(0)} \left( V_{\infty} - V(\e x)\right) |v_{n}|^{p} dx\\
&\leq V_{\infty}\int_{\B_{R}(0)} |v_{n}|^{p} dx + \zeta \int_{\B_{R}^{c}(0)} |v_{n}|^{p} dx\\
&\leq o_{n}(1) + \frac{\zeta}{V_{0}} \int_{\B_{R}^{c}(0)} V(\e x) |v_{n}|^{p} dx \\
&\leq o_{n}(1) + \frac{\zeta}{V_{0}} \|v_{n}\|_{\e}^{p}\leq o_{n}(1)+ \zeta C. 
\end{align*}
Thus, 
\begin{equation}\label{tv10}
\int_{\R^{N}}\int_{\R^{N}} \left[\frac{F(t_{n}v_{n}(y)) f(t_{n}v_{n}(x)) t_{n}v_{n}(x)}{|x-y|^{\mu}}-  \frac{F(v_{n}(y)) f(v_{n}(x)) v_{n}(x)}{|x-y|^{\mu}} \right]    \leq \zeta C +o_{n}(1).
\end{equation}
Since $v_{n} \nrightarrow 0$ in $\h$, we can apply Lemma \ref{lem2.2} to deduce the existence of a sequence $(y_{n})\subset \R^{N}$, and the existence of two positive numbers $\bar{R}, \beta$ such that
\begin{equation}\label{tv11}
\int_{\B_{\bar{R}}(y_{n})} |v_{n}|^{p} dx \geq \beta>0.
\end{equation}
Let us consider $\bar{v}_{n}= v_{n}(x+y_{n})$. Taking into account that $V_{0}<V(\e x)$ and the boundedness of $(v_{n})$ in $\h$, we can see that $(\bar{v}_{n})$ is bounded in $W^{s,p}(\R^{N})$. Then we may assume that $\bar{v}_{n}\rightharpoonup \bar{v}$ in $W^{s,p}(\R^{N})$. By \eqref{tv11} there exists $\Omega \subset \R^{N}$ with positive measure and such that $\bar{v}>0$ in $\Omega$. Using \eqref{tv6}, \eqref{tv10}, and the facts $\frac{f(t)}{t^{\frac{p}{2}-1}}$ and $\frac{F(t)}{t^{\frac{p}{2}}}$ are increasing for $t>0$ in view of $(f_2)$ and $(f_3)$,  we can infer
\begin{align*}
0&<\int_{\Omega} \int_{\Omega} \frac{|v_{n}(x)|^{\frac{p}{2}}|v_{n}(y)|^{\frac{p}{2}}}{|x-y|^{\mu}} \left[\frac{F((1+\delta)v_{n}(y)) f((1+\delta)v_{n}(x)) (1+\delta) v_{n}(x)}{(1+\delta)^{\frac{p}{2}} |v_{n}(x)|^{\frac{p}{2}} (1+\delta)^{\frac{p}{2}} |v_{n}(y)|^{\frac{p}{2}}}-  \frac{F(v_{n}(y)) f(v_{n}(x)) v_{n}(x)}{|v_{n}(y)|^{\frac{p}{2}}|v_{n}(x)|^{\frac{p}{2}}} \right]  \\
&=\int_{\Omega}\int_{\Omega} \left[\frac{F((1+\delta)v_{n}(y)) f((1+\delta)v_{n}(x)) (1+\delta)v_{n}(x)}{(1+\delta)^{p}|x-y|^{\mu}}-  \frac{F(v_{n}(y)) f(v_{n}(x)) v_{n}(x)}{|x-y|^{\mu}} \right]  \\
&\leq \zeta C+o_{n}(1). 
\end{align*}
Taking the limit as $n\rightarrow \infty$ and applying Fatou's Lemma we obtain
\begin{align*}
0<\int_{\Omega}\int_{\Omega} \left[\frac{F((1+\delta)v(y)) f((1+\delta)v(x)) (1+\delta)v(x)}{(1+\delta)^{p}|x-y|^{\mu}}-  \frac{F(v(y)) f(v(x)) v(x)}{|x-y|^{\mu}} \right] \leq \zeta C 
\end{align*}
for any $\zeta>0$, and this is a contradiction. \\
Now, we distinguish the following cases: \\
\textsc{Case 1:} Assume that $\limsup_{n\rightarrow \infty} t_{n}=1$. Then there exists $(t_{n})$ such that $t_{n}\rightarrow 1$. Using  $\J_{\e}(v_{n})\rightarrow d$ and $(t_{n}v_{n})\subset \mathcal{M}_{V_{\infty}}$ we have
\begin{align}\label{tv12new}
d+ o_{n}(1)&= \J_{\e}(v_{n})\nonumber \\
&=\J_{\e}(v_{n}) - \I_{V_{\infty}}(t_{n}v_{n})+ \I_{V_{\infty}}(t_{n}v_{n}) \nonumber \\
&\geq \J_{\e}(v_{n}) -\I_{V_{\infty}}(t_{n}v_{n}) + d_{V_{\infty}}.
\end{align}
Now, we note that 
\begin{align}\begin{split}\label{tv12}
\J_{\e}(v_{n}) &-\I_{V_{\infty}}(t_{n}v_{n}) \\
&= \frac{(1-t_{n}^{p})}{p} [v_{n}]^{p}_{s,p}+ \frac{1}{p} \int_{\R^{N}} \left( V(\e x) - t_{n}^{p} V_{\infty}\right) |v_{n}|^{p} dx + \Sigma(t_{n}v_{n})-\Sigma(v_{n}). 
\end{split} \end{align}
Taking into account assumption $(V)$,  $v_{n}\rightarrow 0$ in $L^{p}(\B_{R}(0))$, $t_{n}\rightarrow 1$, \eqref{tv9}, and 
\begin{align*}
V(\e x) - t_{n}^{p} V_{\infty} =\left(V(\e x) - V_{\infty} \right) + (1- t_{n}^{p}) V_{\infty}\geq -\zeta + (1- t_{n}^{p}) V_{\infty} \quad \forall |x|\geq R
\end{align*}
we get
\begin{align}\label{tv13}
\int_{\R^{N}} & \left( V(\e x) - t_{n}^{p} V_{\infty}\right) |v_{n}|^{p} dx \nonumber \\
&= \int_{\B_{R}(0)} \left( V(\e x) - t_{n}^{p} V_{\infty}\right) |v_{n}|^{p} dx+ \int_{\B_{R}^{c}(0)} \left( V(\e x) - t_{n}^{p} V_{\infty}\right) |v_{n}|^{p} dx \nonumber \\
&\geq (V_{0}- t_{n}^{p}V_{\infty}) \int_{\B_{R}(0)} |v_{n}|^{p} dx - \zeta \int_{\B_{R}^{c}(0)} |v_{n}|^{p} dx+ V_{\infty}(1- t_{n}^{p}) \int_{\B_{R}^{c}(0)} |v_{n}|^{p} dx \nonumber \\
&\geq o_{n}(1)- \zeta C, 
\end{align}
Since $(v_{n})$ is bounded in $\h$ and $t_{n}\rightarrow 1$, we can conclude that 
\begin{align}\label{tv14}
(1-t_{n}^{p}) [v_{n}]^{p}_{s,p}= o_{n}(1). 
\end{align}
Putting together \eqref{tv12}, \eqref{tv13} and \eqref{tv14}, we get
\begin{align}\label{tv15}
\J_{\e}(v_{n})-\I_{V_{\infty}}(t_{n}v_{n}) = \Sigma(t_{n}v_{n})-\Sigma(v_{n}) +o_{n}(1)- \zeta C. 
\end{align}
On the other hand, using Lemma \ref{lemK}, $(f_1)$ and $t_{n}\rightarrow 1$, we get
\begin{align}\label{tv16}
\Sigma(t_{n}v_{n})-\Sigma(v_{n})=\frac{1}{2}\int_{\R^{N}} K(t_{n}v_{n})(F(t_{n}v_{n})-F(v_{n}))+\frac{1}{2}\int_{\R^{N}} K(v_{n})(F(t_{n}v_{n})-F(v_{n}))=o_{n}(1). 
\end{align}
Hence, taking into account \eqref{tv12new}, \eqref{tv15} and \eqref{tv16}, we can infer that 
\begin{align*}
d+o_{n}(1)\geq o_{n}(1) - \zeta C + d_{V_{\infty}}, 
\end{align*}
and taking the limit as $\zeta\rightarrow 0$ we have $d \geq d_{V_{\infty}}$. \\
\textsc{Case 2:} Suppose that $\limsup_{n\rightarrow \infty} t_{n}=t_{0}<1$. Then we can extract a subsequence, still denoted by $(t_{n})$, such that $t_{n}\rightarrow t_{0}<1$ and $t_{n}<1$ for any $n\in \mathbb{N}$. 
Since $\langle \J'_{V_{\infty}}(t_{n}v_{n}), t_{n}v_{n}\rangle=0$ and $\I_{V_{\infty}}(t_{n}v_{n})\geq d_{V_{\infty}}$, and using the fact that $t\mapsto \frac{1}{p}f(t)t- \frac{1}{2} F(t)$ is increasing for $t>0$ by $(f_2)$ and $(f_3)$, we have 
\begin{align*}
d_{V_{\infty}} &\leq \I_{V_{\infty}}(t_{n}v_{n})  \\
&= \I_{V_{\infty}}(t_{n}v_{n}) - \frac{1}{p} \langle \I'_{V_{\infty}}(t_{n}v_{n}), t_{n}v_{n} \rangle \\ 
&= \frac{1}{p} \int_{\R^{N}} \int_{\R^{N}} \frac{F(t_{n}v_{n}(y)) f(t_{n}v_{n}(x))t_{n} v_{n}(x)}{|x-y|^{\mu}}-\frac{1}{2} \int_{\R^{N}} \int_{\R^{N}} \frac{F(t_{n}v_{n}(y)) F(t_{n}v_{n}(x))}{|x-y|^{\mu}} \\
&\leq \frac{1}{p}\int_{\R^{N}} \int_{\R^{N}} \frac{F(v_{n}(y)) f(v_{n}(x)) v_{n}(x)}{|x-y|^{\mu}}-\frac{1}{2} \int_{\R^{N}} \int_{\R^{N}} \frac{F(v_{n}(y)) F(v_{n}(x))}{|x-y|^{\mu}}  \\
&=\J_{\e}(v_{n})-\frac{1}{p} \langle \J'_{\e}(v_{n}), v_{n}\rangle \\
&=d +o_{n}(1).
\end{align*}
Letting the limit as $n\rightarrow \infty$ we can infer that $d\geq d_{V_{\infty}}$.
\end{proof}

\noindent
In view of the previous lemma, we can show that the Palais-Smale condition holds in a suitable sublevel, related to the ground energy at infinity.
\begin{prop}\label{prop2.1}
Let $(u_{n})\subset \N_{\e}$ be such that $\J_{\e}(u_{n})\rightarrow c$, where $c<d_{V_{\infty}}$ if $V_{\infty}<\infty$ and $c\in \R$ if $V_{\infty}=\infty$. Then $(u_{n})$ has a convergent subsequence in $\h$.
\end{prop}

\begin{proof}
Arguing as in Lemma \ref{lemB} we can see that $(u_{n})$ is bounded in $\h$. Then, up to a subsequence, we may assume that
\begin{align}\begin{split}\label{conv}
&u_{n}\rightharpoonup u \mbox{ in } \h, \\
&u_{n}\rightarrow u \mbox{ in } L^{q}_{loc}(\R^{N}) \quad \mbox{ for any } q\in [1, \p), \\
&u_{n} \rightarrow u \mbox{ a.e. in } \R^{N}.
\end{split}\end{align}
Using $(f_{1})$, \eqref{conv} and the fact that $C^{\infty}_{c}(\R^{N})$ is dense in $W^{s, p}(\R^{N})$, it is standard to check that $\J'_{\e}(u)=0$.
Now, let $v_{n}= u_{n}-u$.
In view of Lemma \ref{lemPSY} and Lemma \ref{BLlem} we can see that
\begin{align}\label{tv19}
\J_{\e}(v_{n})&=\frac{\|u_{n}\|_{\e}^{p}}{p} - \frac{\|u\|_{\e}^{p}}{p}- \Sigma(u_{n})+ \Sigma(u) + o_{n}(1) \nonumber \\
&=\J_{\e}(u_{n})-\J_{\e}(u)+o_{n}(1) \nonumber\\
&=c-\J_{\e}(u)+o_{n}(1)=:d+o_{n}(1).
\end{align}
Moreover, we can prove that $\J'_{\e}(v_{n})=o_{n}(1)$. Indeed, applying Lemma \ref{lemVince} and Lemma $3.3$ in \cite{MeW} with $z_{n}= v_{n}$ and $w=u$ we get
\begin{equation}\label{D}
\iint_{\R^{2N}} |\mathcal{A}(u_{n}) - \mathcal{A}(v_{n}) - \mathcal{A}(u)|^{p'} dx= o_{n}(1),
\end{equation}
and 
\begin{equation}\label{C}
\int_{\R^{N}} V(\e x) ||v_{n}|^{p-2}v_{n}-|u_{n}|^{p-2}u_{n}+|u|^{p-2}u|^{p'} dx=o_{n}(1).
\end{equation}
From the H\"older inequality, we have for any $\varphi\in \h$ such that $\|\varphi\|_{\e}\leq 1$
\begin{align*}
&|\langle \J'_{\e}(v_{n})-\J'_{\e}(u_{n})+\J'_{\e}(u), \varphi\rangle| \\
&\leq \left(\iint_{\R^{2N}}  |\mathcal{A}(u_{n}) - \mathcal{A}(v_{n}) - \mathcal{A}(u)|^{p'} dx dy\right)^{\frac{1}{p'}} [\varphi]_{s,p} \\
&+\left(\int_{\R^{N}} V(\e x) ||v_{n}|^{p-2}v_{n}-|u_{n}|^{p-2}u_{n}+|u|^{p-2}u|^{p'} dx\right)^{p'} \left(\int_{\R^{N}} V(\e x) |\varphi|^{p}dx \right)^{\frac{1}{p}} \\
&+|\langle \Sigma'(v_{n})-\Sigma'(u_{n})+\Sigma'(u), \varphi \rangle|,
\end{align*}
and in view of Lemma \ref{BLlem}, \eqref{D}, \eqref{C}, $\J'_{\e}(u_{n})=0$ and $\J'_{\e}(u)=0$ we obtain that $\langle \J'_{\e}(u_{n}), \varphi\rangle =o_{n}(1)$ for any $\varphi\in \h$ such that $\|\varphi\|_{\e}\leq 1$.

On the other hand, using  $(f_2)$, we can see that
\begin{equation}\label{tv199}
\J_{\e}(u)=\J_{\e}(u)-\frac{1}{p} \langle\J'_{\e}(u),u\rangle=\frac{1}{p}\int_{\R^{N}} \int_{\R^{N}} \frac{F(u(y)) f(u(x)) u(x)}{|x-y|^{\mu}}-\frac{1}{2} \int_{\R^{N}} \int_{\R^{N}} \frac{F(u(y)) F(u(x))}{|x-y|^{\mu}}\geq 0.
\end{equation}
Now, suppose that $V_{\infty}<\infty$. From \eqref{tv19} and \eqref{tv199} it follows that
$$
d\leq c<d_{V_{\infty}}
$$
which together with Lemma \ref{lem2.3} implies that $v_{n}\rightarrow 0$ in $\h$, that is $u_{n}\rightarrow u$ in $\h$.

Finally, we deal with the case $V_{\infty}=\infty$. Then, by Theorem \ref{Cheng}, we deduce that $v_{n}\rightarrow 0$ in $L^{r}(\R^{N})$ for all $r\in [p, \p)$. This fact combined with $(f_1)$ and Lemma \ref{lemK} yields
\begin{equation}\label{fn0}
\int_{\R^{N}} K(v_{n}) f(v_{n})v_{n}dx=o_{n}(1).
\end{equation}
In view of $\langle\J'_{\e}(v_{n}), v_{n}\rangle=o_{n}(1)$ and  \eqref{fn0} we can infer that
$$
\|v_{n}\|^{p}_{\e}=o_{n}(1),
$$
which gives $u_{n}\rightarrow u$ in $\h$.
\end{proof}

\noindent
Now, we are ready to prove the main result of this section.
\begin{thm}\label{thm3.1}
Assume that $(V)$ and $(f_1)$-$(f_3)$ hold. Then there exists $\e_{0}>0$ such that problem \eqref{R} admits a ground state solution for any $\e\in (0, \e_{0})$.
\end{thm}
\begin{proof}
From $(v)$ of Lemma \ref{SW2}, we know that $c_{\e}\geq \rho>0$ for each $\e>0$. Moreover, if $u\in \mathcal{N}_{\e}$ verifies $\J_{\e}(u)=c_{\e}$, then $m_{\e}^{-1}(u)$ is a minimizer of $\Psi_{\e}$ and it is a critical point of $\Psi_{\e}$. In view of Lemma \ref{SW3} we can see that $u$ is a critical point of $\J_{\e}$.

Now we show that there exists a minimizer of $\J_{\e}|_{\mathcal{N}_{\e}}$. Applying Ekeland's variational principle  there exists a sequence $(v_{n})\subset \mathbb{S}_{\e}$ such that $\Psi_{\e}(v_{n})\rightarrow c_{\e}$ and $\Psi'_{\e}(v_{n})\rightarrow 0$ as $n\rightarrow \infty$. Let $u_{n}=m_{\e}(v_{n}) \in \mathcal{N}_{\e}$. Then, from Lemma \ref{SW3} we deduce that $\J_{\e}(u_{n})\rightarrow c_{\e}$, $\langle \J'_{\e}(u_{n}), u_{n}\rangle =0$ and $\J'_{\e}(u_{n})\rightarrow 0$ as $n\rightarrow \infty$.
Therefore, $(u_{n})$ is a Palais-Smale sequence for $\J_{\e}$ at level $c_{\e}$.
It is standard to check that $(u_{n})$ is bounded in $\h$ and we denote by $u$ its weak limit. It is easy to verify that $\J_{\e}'(u)=0$.
Let us consider $V_{\infty}=\infty$. Using Lemma \ref{Cheng}, we have $\J_{\e}(u)=c_{\e}$ and $\J'_{\e}(u)=0$.

Now, we deal with the case $V_{\infty}<\infty$. In view of Proposition \ref{prop2.1} it is enough to show that $c_{\e}<d_{V_{\infty}}$ for $\e>0$ small enough. Without loss of generality, we may suppose that 
$$
V(0)=V_{0}=\inf_{x\in \R^{N}} V(x).
$$
Let $\mu\in \R$ be such that $\mu\in (V_{0}, V_{\infty})$. Then we can see that $d_{V_{0}}<d_{\mu}<d_{V_{\infty}}$. 
Let $\eta_{r}\in C^{\infty}_{c}(\R^{N})$ be a cut-off function such that $\eta_{r}=1$ in $\B_{r}(0)$ and $\eta_{r}=0$ in $\B_{2r}^{c}(0)$. Let us define $w_{r}(x):=\eta_{r}(x) w(x)$, where $w\in W^{s,p}(\R^{N})$ is a positive ground state to autonomous problem $(P_{\mu})$, which there exists by Lemma \ref{lem4.3}. Take $t_{r}>0$ such that 
\begin{equation*}
\I_{\mu}(t_{r} w_{r})=\max_{t\geq 0} \I_{\mu}(t w_{r}). 
\end{equation*}
Our next claim consists in finding $r$ sufficiently large such that $\I_{\mu}(t_{r} w_{r})<d_{V_{\infty}}$. \\
Assume by contradiction $\I_{\mu}(t_{r} w_{r})\geq d_{V_{\infty}}$ for any $r>0$. 
Taking into account $w_{r}\rightarrow w$ in $W^{s,p}(\R^{N})$ as $r\rightarrow \infty$ in view of Lemma \ref{pp}, $t_{r}w_{r}$ and $w$ belong to $\mathcal{M}_{\mu}$ and using assumption $(f_{3})$, we have  $t_{r}\rightarrow 1$, and
$$
d_{V_{\infty}}\leq \liminf_{r\rightarrow \infty} \I_{\mu}(t_{r} w_{r})=\I_{\mu}(w)=d_{\mu}
$$
which is impossible since $d_{V_{\infty}}>d_{\mu}$.
Hence, there exists $r>0$ such that
\begin{align}\label{tv200}
\I_{\mu}(t_{r} w_{r})=\max_{\tau\geq 0} \I_{\mu}(\tau (t_{r} w_{r}))\quad  \mbox{ and } \quad \I_{\mu}(t_{r}w_{r})<d_{V_{\infty}}.
\end{align}
Now, condition $(V)$ implies that there exists  $\e_{0}>0$ such that 
\begin{equation}\label{tv20}
V(\e x)\leq \mu \mbox{ for all } x\in \supp(w_{r}), \e\in (0, \e_{0}).
\end{equation} 
Therefore, by \eqref{tv200} and \eqref{tv20}, we deduce that
$$
c_{\e}\leq \max_{\tau\geq 0} \J_{\e}(\tau (t_{r} w_{r}))\leq \max_{\tau\geq 0} \I_{\mu}(\tau (t_{r} w_{r}))=\I_{\mu}(t_{r} w_{r})<d_{V_{\infty}}
$$ 
which implies that $c_{\e}<d_{V_{\infty}}$ for any $\e>0$ sufficiently small.
\end{proof}

\section{Proof of Theorem \ref{thmA1}}
In this last section we investigate the multiplicity of solutions to (\ref{P}). First of all, we need to introduce some useful tools.\\
Fix $\delta>0$, and let $w$ be a ground state solution for $(P_{V_{0}})$ (whose existence is guaranteed by Lemma \ref{lem4.3}). 
Let $\eta$ be a smooth nonincreasing cut-off function defined on $[0, \infty)$ satisfying $\eta(t)=1$ if $0\leq t\leq \frac{\delta}{2}$ and $\eta(t)=0$ if $t\geq \delta$.\\
For any $y\in M$, we define
$$
\Psi_{\e, y}(x)=\eta(|\e x-y|) w\left(\frac{\e x-y}{\e}\right),
$$
and we denote by $t_{\e}>0$ the unique positive number such that 
$$
\max_{t\geq 0} \J_{\e}(t \Psi_{\e, y})=\J_{\e}(t_{\e} \Psi_{\e, y}).
$$
Finally, we consider $\Phi_{\e}: M\rightarrow \N_{\e}$ defined as $\Phi_{\e}(y)=t_{\e} \Psi_{\e, y}$.

\begin{lem}\label{lemma3.4}
The functional $\Phi_{\e}$ satisfies the following limit
$$
\lim_{\e\rightarrow 0} \J_{\e}(\Phi_{\e}(y))=d_{V_{0}} \mbox{ uniformly in } y\in M.
$$
\end{lem}
\begin{proof}
Assume by contradiction that there exist $\delta_{0}>0$, $(y_{n})\subset M$ and $\e_{n}\rightarrow 0$ such that 
\begin{equation}\label{4.41}
|\J_{\e_{n}}(\Phi_{\e_{n}}(y_{n}))-d_{V_{0}}|\geq \delta_{0}.
\end{equation}
Let us note that using the change of variable $z=\frac{\e_{n}x-y_{n}}{\e_{n}}$, we have
\begin{align}\label{defJ}
\J_{\e_{n}}(\Phi_{\e_{n}}(y_{n}))&=\frac{t_{\e_{n}}^{p}}{p}[\eta(|\e_{n} \cdot|)w]^{p}_{s,p}+\frac{t_{\e_{n}}^{p}}{p}\int_{\R^{N}} V(\e_{n} z+y_{n}) (\eta(|\e_{n} z|) w(z))^{p}\, dz \nonumber\\
&-\Sigma(t_{\e_{n}}\eta(|\e_{n} \cdot|)w).
\end{align}
In view of the Dominated Convergence Theorem and Lemma \ref{pp}, it is easy to check that
$$
\lim_{n\rightarrow \infty} \|\Psi_{\e_{n}, y_{n}}\|_{\e_{n}}=\|w\|_{V_{0}}\in (0, \infty)
$$
and
$$
\lim_{n\rightarrow \infty} \Sigma(\Psi_{\e_{n}, y_{n}})=\Sigma(w).
$$

Since $t_{\e_{n}}\Psi_{\e_{n}, y_{n}}\in \mathcal{N}_{\e_{n}}$, we can see that 
\begin{equation*}
t_{\e_{n}}^{p}\|\Psi_{\e_{n}, y_{n}}\|_{\e_{n}}^{p}=\int_{\R^{N}}\int_{\R^{N}} \frac{F(t_{\e_{n}}\Psi_{\e_{n}, y_{n}})f(t_{\e_{n}}\Psi_{\e_{n}, y_{n}})t_{\e_{n}}\Psi_{\e_{n}, y_{n}}}{|x-y|^{\mu}},
\end{equation*}
so we can deduce that
\begin{align}\label{wV0}
\|w\|^{p}_{V_{0}}=\lim_{n\rightarrow \infty} \int_{\R^{N}}\int_{\R^{N}} \frac{F(t_{\e_{n}}\Psi_{\e_{n}, y_{n}})f(t_{\e_{n}}\Psi_{\e_{n}, y_{n}})t_{\e_{n}}\Psi_{\e_{n}, y_{n}}}{t_{\e_{n}}^{p}|x-y|^{\mu}}.
\end{align}
Taking into account $\frac{f(t)}{t^{\frac{p}{2}-1}}$ and $\frac{F(t)}{t^{\frac{p}{2}}}$ are increasing for $t>0$, $\eta=1$ in $\B_{\frac{\delta}{2}}(0)$ and $\B_{\frac{\delta}{2}}(0)\subset \B_{\frac{\delta}{2\e_{n}}}(0)$ for all $n$ big enough, we obtain
\begin{align}\label{delta}
\int_{\R^{N}}\int_{\R^{N}} \frac{F(t_{\e_{n}}\Psi_{\e_{n}, y_{n}})f(t_{\e_{n}}\Psi_{\e_{n}, y_{n}})\Psi_{\e_{n}, y_{n}}}{t^{p-1}_{\e_{n}}|x-y|^{\mu}}\geq |\B_{\frac{\delta}{2}}(0)|\frac{F(t_{\e_{n}}w(\bar{z}))}{(t_{\e_{n}}w(\bar{z}))^{\frac{p}{2}}} \frac{f(t_{\e_{n}}w(\bar{z}))}{(t_{\e_{n}}w(\bar{z}))^{\frac{p}{2}-1}}\int_{\B_{\frac{\delta}{2}}(0)} w(z)^{p}
\end{align}
where $w(\bar{z})=\min_{z\in \overline{\B}_{\frac{\delta}{2}}(0)} w(z)>0$ (we recall that $w\in C(\R^{N})$ and $w>0$ in $\R^{N}$ by Lemma \ref{lem4.3}). \\
Hence, if $t_{\e_{n}}\rightarrow \infty$, we can use $(f_2)$ to see that
$$
\lim_{t\rightarrow \infty} \frac{F(t)}{t^{\frac{p}{2}}}=\lim_{t\rightarrow \infty} \frac{f(t)}{t^{\frac{p}{2}-1}}=\infty
$$
which together with \eqref{wV0} and \eqref{delta} gives a contradiction. Then, there exists $t_{0}\geq 0$ such that $t_{\e_{n}}\rightarrow t_{0}\geq 0$. In particular, from $(f_1)$ and \eqref{wV0}, we can see that $t_{0}>0$.\\
Recalling that $w$ is a ground state to $(P_{V_{0}})$ and using the fact that the maps $\frac{f(t)}{t^{\frac{p}{2}-1}}$ and $\frac{F(t)}{t^{\frac{p}{2}}}$ are increasing for $t>0$, we can conclude that  $t_{\e_{n}}\rightarrow t_{0}=1$.
Accordingly,
$$
\lim_{n\rightarrow \infty} \Sigma(t_{\e_{n}}\eta(|\e_{n} \cdot|)w)=\Sigma(w)
$$
and passing to the limit in \eqref{defJ}, we get
$$
\lim_{n\rightarrow \infty} \J_{\e_{n}}(\Phi_{\e_{n}}(y_{n}))=\I_{V_{0}}(w)=d_{V_{0}},
$$
which contradicts (\ref{4.41}).

\end{proof}

\noindent
For any $\delta>0$, let $\rho=\rho(\delta)>0$ be such that $M_{\delta}\subset \B_{\rho}(0)$.
Let $\Upsilon: \R^{N}\rightarrow \R^{N}$ be defined as $\Upsilon(x)=x$ for $|x|\leq \rho$ and $\Upsilon(x)=\frac{\rho x}{|x|}$ for $|x|\geq \rho$.
Then, we consider the barycenter map $\beta_{\e}: \mathcal{N}_{\e}\rightarrow \R^{N}$ given by
$$
\beta_{\e}(u)=\frac{\int_{\R^{N}} \Upsilon(\e x) |u(x)|^{p}\, dx}{\int_{\R^{N}} |u(x)|^{p} \,dx}.
$$

\begin{lem}\label{lemma3.5}
The function $\beta_{\e}$ verifies the following limit
$$
\lim_{\e \rightarrow 0} \beta_{\e}(\Phi_{\e}(y))=y \mbox{ uniformly in } y\in M.
$$
\end{lem}
\begin{proof}
Suppose by contradiction that there exist $\delta_{0}>0$, $(y_{n})\subset M$ and $\e_{n}\rightarrow 0$ such that 
\begin{equation}\label{4.4}
|\beta_{\e_{n}}(\Phi_{\e_{n}}(y_{n}))-y_{n}|\geq \delta_{0}.
\end{equation}
From the definitions of $\Phi_{\e_{n}}(y_{n})$, $\beta_{\e_{n}}$ and $\eta$, we can see that 
$$
\beta_{\e_{n}}(\Phi_{\e_{n}}(y_{n}))=y_{n}+\frac{\int_{\R^{N}}[\Upsilon(\e_{n}z+y_{n})-y_{n}] |\eta(|\e_{n}z|) w(z)|^{p} \, dz}{\int_{\R^{N}} |\eta(|\e_{n}z|) w(z)|^{p}\, dz}.
$$
Recalling that $(y_{n})\subset M\subset \B_{\rho}(0)$ and using the Dominated Convergence Theorem, we obtain that 
$$
|\beta_{\e_{n}}(\Phi_{\e_{n}}(y_{n}))-y_{n}|=o_{n}(1)
$$
which contradicts (\ref{4.4}).

\end{proof}

At this point, we prove the following compactness result which will be crucial in the sequel.
\begin{lem}\label{prop3.3}
Let $\e_{n}\rightarrow 0^{+}$ and $(u_{n})\subset \mathcal{N}_{\e_{n}}$ be such that $\J_{\e_{n}}(u_{n})\rightarrow d_{V_{0}}$. Then there exists $(\tilde{y}_{n})\subset \R^{N}$ such that $v_{n}(x)=u_{n}(x+\tilde{y}_{n})$ has a convergent subsequence in $W^{s,p}(\R^{N})$. Moreover, up to a subsequence, $y_{n}=\e_{n} \tilde{y}_{n}\rightarrow y\in M$.
\end{lem}
\begin{proof}
Since $\langle \J'_{\e_{n}}(u_{n}), u_{n}\rangle=0$ and $\J_{\e_{n}}(u_{n})\rightarrow d_{V_{0}}$, we can argue as in Lemma \ref{lemB} to see that $(u_{n})$ is bounded in $\mathcal{W}_{\e_{n}}$. 
Now, we show that there exist a sequence $(\tilde{y}_{n})\subset \R^{N}$, and constants $R>0$ and $\beta>0$ such that
\begin{equation}\label{sacchi}
\liminf_{n\rightarrow \infty}\int_{\B_{R}(\tilde{y}_{n})} |u_{n}|^{p} \, dx\geq \beta>0.
\end{equation}
Suppose that condition \eqref{sacchi} does not hold. Then, for all $R>0$, we have
$$
\lim_{n\rightarrow \infty}\sup_{y\in \R^{N}}\int_{\B_{R}(y)} |u_{n}|^{p} \, dx=0.
$$
Since we know that $(u_{n})$ is bounded in $W^{s,p}(\R^{N})$, we can use Lemma \ref{Lions} to deduce that $u_{n}\rightarrow 0$ in $L^{q}(\R^{N})$ for any $q\in (p, p^{*}_{s})$. 
Taking into account $\langle \J'_{\e_{n}}(u_{n}), u_{n}\rangle=0$ and applying Theorem \ref{HLS} and $(f_1)$, we can  infer that $\|u_{n}\|_{\e_{n}}\rightarrow 0$ as $n\rightarrow \infty$. Then, $\J_{\e_{n}}(u_{n})\rightarrow 0$ as $n\rightarrow \infty$, and this is a contradiction because of $\J_{\e_{n}}(u_{n})\rightarrow d_{V_{0}}>0$. 

Now, we set $v_{n}(x)=u_{n}(x+\tilde{y}_{n})$. Then, $(v_{n})$ is bounded in $W^{s,p}(\R^{N})$, and we may assume that 
$v_{n}\rightharpoonup v\not\equiv 0$ in $W^{s,p}(\R^{N})$ as $n\rightarrow \infty$.
Fix $t_{n}>0$ such that $\tilde{v}_{n}=t_{n} v_{n}\in \mathcal{M}_{V_{0}}$. Since $u_{n}\in \mathcal{N}_{\e_{n}}$, we can see that 
$$
d_{V_{0}}\leq \I_{V_{0}}(\tilde{v}_{n})= \I_{V_{0}}(t_{n}u_{n})\leq \J_{\e_{n}}(t_{n}u_{n})\leq \J_{\e_{n}}(u_{n})= d_{V_{0}}+o_{n}(1)
$$
which gives $\I_{V_{0}}(\tilde{v}_{n})\rightarrow d_{V_{0}}$. From Ekeland's variational principle, we may assume that $(\tilde{v}_{n})$ is a bounded $(PS)_{d_{V_{0}}}$.
In particular, we get $\tilde{v}_{n}\rightharpoonup \tilde{v}$ in $W^{s,p}(\R^{N})$ for some $\tilde{v}\not\equiv 0$, and $\I'_{V_{0}}(\tilde{v})=0$. \\
Using Lemma \ref{lemVince} and \ref{BLlem}, we can deduce that
$$
\I_{V_{0}}(\tilde{v}_{n}-\tilde{v})\rightarrow d_{V_{0}}-\I_{V_{0}}(\tilde{v}) \mbox{ and } \I'_{V_{0}}(\tilde{v}_{n}-\tilde{v})\rightarrow 0.
$$
Since Fatou's Lemma gives
\begin{align*}
d_{V_{0}}=\lim_{n\rightarrow \infty} \I_{V_{0}}(\tilde{v}_{n})&=\lim_{n\rightarrow \infty} \left(\frac{1}{p}\langle \Sigma'(\tilde{v}_{n}), \tilde{v}_{n}\rangle-\Sigma(\tilde{v}_{n}) \right) \\
&\geq \left( \frac{1}{p}\langle \Sigma'(\tilde{v}), \tilde{v}\rangle-\Sigma(\tilde{v}) \right)=\I_{V_{0}}(\tilde{v})\geq d_{V_{0}},
\end{align*}
we can infer that
$$
\I_{V_{0}}(\tilde{v}_{n}-\tilde{v})\rightarrow 0 \mbox{ and } \I'_{V_{0}}(\tilde{v}_{n}-\tilde{v})\rightarrow 0.
$$
Therefore, 
$$
\|\tilde{v}_{n}-\tilde{v}\|^{p}_{V_{0}}\leq C \left[\I_{V_{0}}(\tilde{v}_{n}-\tilde{v})-\frac{1}{\theta}\langle \I'_{V_{0}}(\tilde{v}_{n}-\tilde{v}), \tilde{v}_{n}-\tilde{v} \rangle \right] \rightarrow 0,
$$ 
that is
\begin{equation}\label{elena}
\tilde{v}_{n}\rightarrow \tilde{v} \mbox{ in } W^{s,p}(\R^{N}).
\end{equation} 
This and the fact that $t_{n}\rightarrow t_{0}$, for some $t_{0}>0$, yield $v_{n}\rightarrow v$ in $W^{s,p}(\R^{N})$ as $n\rightarrow \infty$.

Now, we set $y_{n}=\e_{n}\tilde{y}_{n}$, and we aim to prove that $(y_{n})$ admits a subsequence, still denoted by $y_{n}$, such that $y_{n}\rightarrow y$, for some $y\in M$. Firstly, we show that $(y_{n})$ is bounded. We argue by contradiction, and we assume that, up to a subsequence, $|y_{n}|\rightarrow \infty$ as $n\rightarrow \infty$. 
Taking into account \eqref{elena} and $V_{0}<V_{\infty}$, we get
\begin{align*}
d_{V_{0}}&=\I_{V_{0}}(\tilde{v})<\I_{V_{\infty}}(\tilde{v}) \\
&\leq \liminf_{n\rightarrow \infty} \left[\frac{1}{p}[\tilde{v}_{n}]^{p}_{s,p}+\frac{1}{2}\int_{\R^{N}} V(\e_{n}x+y_{n})|\tilde{v}_{n}|^{p}-\Sigma(\tilde{v}_{n})  \right] \\
&=\liminf_{n\rightarrow \infty} \left[\frac{t_{n}^{p}}{p}[u_{n}]^{p}_{s,p}+\frac{t_{n}^{p}}{p}\int_{\R^{N}} V(\e_{n}z)|u_{n}|^{p}-\Sigma(t_{n}u_{n})  \right] \\
&=\liminf_{n\rightarrow \infty} \J_{\e}(t_{n} u_{n}) \\
&\leq \liminf_{n\rightarrow \infty} \J_{\e}(u_{n})=d_{V_{0}}
\end{align*}
which gives an absurd.
Therefore, $(y_{n})$ is bounded, and we may assume that $y_{n}\rightarrow y\in \R^{N}$. Clearly, $y\in M$ otherwise we can argue as above to get a contradiction.
\end{proof}

\noindent
Now, we define a map $h:\R_{+}\rightarrow \R_{+}$ given by $h(\e)=\max_{y\in M}|\J_{\e}(\Phi_{\e}(y))-d_{V_{0}}|$. By Lemma \ref{lemma3.4}, we know that $h(\e)\rightarrow 0$.
Let us consider
$$
\widetilde{\mathcal{N}}_{\e}=\{u\in \mathcal{N}_{\e}: \J_{\e}(u)\leq d_{V_{0}}+h(\e)\},
$$
and we note that $\widetilde{\mathcal{N}}_{\e}\neq \emptyset$  because  $\Phi_{\e}(y)\in \widetilde{\mathcal{N}}_{\e}$ for all $y\in M$. Moreover, we can see that
\begin{lem}\label{lemma3.7}
$$
\lim_{\e \rightarrow 0} \sup_{u\in \widetilde{\mathcal{N}}_{\e}} dist(\beta_{\e}(u), M_{\delta})=0.
$$
\end{lem}
\begin{proof}
Let $\e_{n}\rightarrow 0$ as $n\rightarrow \infty$. For any $n\in \mathbb{N}$, there exists $u_{n}\in \widetilde{\mathcal{N}}_{\e_{n}}$ such that
$$
\sup_{u\in \tilde{\mathcal{N}}_{\e_{n}}} \inf_{y\in M_{\delta}}|\beta_{\e_{n}}(u)-y|=\inf_{y\in M_{\delta}}|\beta_{\e_{n}}(u_{n})-y|+o_{n}(1).
$$
Therefore, it is suffices to prove that there exists $(y_{n})\subset M_{\delta}$ such that 
\begin{equation}\label{3.13}
\lim_{n\rightarrow \infty} |\beta_{\e_{n}}(u_{n})-y_{n}|=0.
\end{equation}
We note that $(u_{n})\subset  \widetilde{\mathcal{N}}_{\e_{n}}\subset  \mathcal{N}_{\e_{n}}$, from which we deduce that
$$
d_{V_{0}}\leq c_{\e_{n}}\leq \J_{\e_{n}}(u_{n})\leq d_{V_{0}}+h(\e_{n}).
$$
This yields $\J_{\e_{n}}(u_{n})\rightarrow d_{V_{0}}$. Using Lemma \ref{prop3.3}, there exists $(\tilde{y}_{n})\subset \R^{N}$ such that $y_{n}=\e_{n}\tilde{y}_{n}\in M_{\delta}$ for $n$ sufficiently large. Setting $v_{n}=u_{n}(\cdot+\tilde{y}_{n})$ and using a change of variable, we can see that
$$
\beta_{\e_{n}}(u_{n})=y_{n}+\frac{\int_{\R^{N}}[\Upsilon(\e_{n}x+y_{n})-y_{n}] |v_{n}|^{p} \, dx}{\int_{\R^{N}} |v_{n}|^{p}\, dx}.
$$
Since $\e_{n} x+y_{n}\rightarrow y\in M$, we deduce that $\beta_{\e_{n}}(u_{n})=y_{n}+o_{n}(1)$, that is (\ref{3.13}) holds.

\end{proof}

\subsection{Multiple solutions to \eqref{P}}
\noindent
Now we show that \eqref{R} admits at least $cat_{M_{\delta}}(M)$ positive solutions.
In order to achieve our aim, we recall the following result for critical points involving Ljusternik-Schnirelmann category; see \cite{MW}.
\begin{thm}\label{LSt}
Let $U$ be a $C^{1,1}$ complete Riemannian manifold (modeled on a Hilbert space). Assume that $h\in C^{1}(U, \R)$ bounded from below and satisfies $-\infty<\inf_{U} h<d<k<\infty$. Moreover, suppose that $h$ satisfies Palais-Smale condition on the sublevel $\{u\in U: h(u)\leq k\}$ and that $d$ is not a critical level for $h$. Then
$$
card\{u\in h^{d}: \nabla h(u)=0\}\geq cat_{h^{d}}(h^{d}).
$$
\end{thm}

Since $\mathcal{N}_{\e}$ is not a $C^{1}$ submanifold of $\h$, we cannot directly apply Theorem \ref{LSt}. However, in view of Lemma \ref{SW2}, we know that the mapping $m_{\e}$ is a homeomorphism between $\mathcal{N}_{\e}$ and $\mathbb{S}_{\e}$, and $\mathbb{S}_{\e}$ is a $C^{1}$ submanifold of $\h$. Therefore, we can apply Theorem \ref{LSt} to
$\Psi_{\e}(u)=\J_{\e}(\hat{m}_{\e}(u))|_{\mathbb{S}_{\e}}=\J_{\e}(m_{\e}(u))$, where $\Psi_{\e}$ is given in Lemma \ref{SW3}.
\begin{thm}\label{teorema}
Assume that $(V)$ and $(f_1)$-$(f_3)$ hold. Then, for any $\delta>0$ there exists $\bar{\e}_\delta>0$ such that, for any $\e \in (0, \bar{\e}_\delta)$, problem \eqref{R} has at least $cat_{M_{\delta}}(M)$ positive solutions.
\end{thm}

\begin{proof}
For any $\e>0$, we define $\alpha_{\e} : M \rightarrow \mathbb{S}_{\e}$ by $\alpha_{\e}(y)= m_{\e}^{-1}(\Phi_{\e}(y))$. Using Lemma \ref{lemma3.4} and the definition of $\Psi_{\e}$, we can see that
\begin{equation*}
\lim_{\e \rightarrow 0} \Psi_{\e}(\alpha_{\e}(y)) = \lim_{\e \rightarrow 0} \J_{\e}(\Phi_{\e}(y))= d_{V_{0}} \quad \mbox{ uniformly in } y\in M.
\end{equation*}
Thus, there exists $\tilde{\e}>0$ such that $\tilde{\mathbb{S}}_{\e}:=\{ w\in \mathbb{S}_{\e} : \Psi_{\e}(w) \leq d_{V_{0}} + h(\e)\} \neq \emptyset$ for all $\e \in (0, \tilde{\e})$, where $h(\e)= |\Psi_{\e}(\alpha_{\e}(y)) - d_{V_{0}}|\rightarrow 0$ as $\e\rightarrow 0$.

Putting together Lemma \ref{SW2}, Lemma \ref{SW3}, Lemma \ref{lemma3.4}, Lemma \ref{lemma3.5} and Lemma \ref{lemma3.7}, we can find $\bar{\e}= \bar{\e}_{\delta}>0$ such that the following diagram
\begin{equation*}
M\stackrel{\Phi_{\e}}{\rightarrow} \widetilde{\mathcal{N}}_{\e} \stackrel{m_{\e}^{-1}}{\rightarrow} \tilde{\mathbb{S}}_{\e} \stackrel{m_{\e}}{\rightarrow} \widetilde{\mathcal{N}}_{\e} \stackrel{\beta_{\e}}{\rightarrow} M_{\delta}
\end{equation*}
is well defined for any $\e \in (0, \bar{\e})$.

In view of Lemma \ref{lemma3.5}, there exists a function $\theta(\e, y)$ with $|\theta(\e, y)|<\frac{\delta}{2}$ uniformly in $y\in M$ for all $\e \in (0, \bar{\e})$ such that $\beta_{\e}(\Phi_{\e}(y))= y+ \theta(\e, y)$ for all $y\in M$. Then, we can see that $H(t, y)= y+ (1-t)\theta(\e, y)$ with $(t, y)\in [0,1]\times M$ is a homotopy between $\beta_{\e} \circ \Phi_{\e}=(\beta_{\e}\circ m_{\e}) \circ \alpha_{\e}$ and the inclusion map $id: M \rightarrow M_{\delta}$. This fact implies that $cat_{\tilde{\mathbb{S}}_{\e}} (\tilde{\mathbb{S}}_{\e})\geq cat_{M_{\delta}}(M)$.

On the other hand, let us choose a function $h(\e)>0$ such that $h(\e)\rightarrow 0$ as $\e\rightarrow 0$ and such that $d_{V_{0}}+h(\e)$ is not a critical level for $\J_{\e}$. For $\e>0$ small enough, we deduce from Proposition \ref{prop2.1} that $\J_{\e}$ satisfies the Palais-Smale condition in $\widetilde{\N}_{\e}$. Then, by $(ii)$ of Lemma \ref{SW3}, we infer that $\Psi_{\e}$ satisfies the Palais-Smale condition in $\tilde{\mathbb{S}}_{\e}$. Hence, by Theorem \ref{LSt}, we obtain that $\Psi_{\e}$ has at least $cat_{\tilde{\mathbb{S}}_{\e}}(\tilde{\mathbb{S}}_{\e})$ critical points on $\tilde{\mathbb{S}}_{\e}$. In the light of $(iii)$ of Lemma \ref{SW3}, we can infer that $\J_{\e}$ admits at least $cat_{M_{\delta}}(M)$ critical points. 
\end{proof}

\subsection{Concentration of the maximum points}
In what follows, we study the behavior of maximum points of solutions to \eqref{R}. Firstly, we establish $L^{\infty}$-estimate using a variant of the Moser iteration argument  \cite{Moser}.

\begin{lem}\label{lemMoser}
Let $v_{n}$ be a solution to
\begin{equation}\label{Pn}
\left\{
\begin{array}{ll}
(-\Delta)^{s}_{p} v_{n} + V_{n}(x) |v_{n}|^{p-2}v_{n}= \left(\frac{1}{|x|^{\mu}}*F(v_{n})\right)f(v_{n}) &\mbox{ in } \R^{N}\\
v_{n}\in W^{s,p}(\R^{N}),\quad v_{n}>0 &\mbox{ in } \R^{N}, 
\end{array}
\right.
\end{equation}
where $V_{n}(x)= V(\e_{n}x+ \e_{n}\tilde{y}_{n})$, and $\e_{n}\tilde{y}_{n}\rightarrow y\in M$. 

If $v_{n}\rightarrow v\neq 0$ in $W^{s,p}(\R^{N})$, then there exists $C>0$ such that 
$$
|v_{n}|_{\infty}\leq C \mbox{ for all } n\in \mathbb{N}.
$$ 
Furthermore 
\begin{equation*}
\lim_{|x|\rightarrow \infty} v_{n}(x)=0 \mbox{ uniformly in } n\in \mathbb{N},
\end{equation*}
and there exists $\sigma>0$ such that $|v_{n}|_{\infty}\geq \sigma$ for any $n\in \mathbb{N}$.

\end{lem}
\begin{proof}
For any $L>0$ and $\beta>1$, let us consider the function 
\begin{equation*}
\gamma(v_{n})=\gamma_{L, \beta}(v_{n})=v_{n} v_{L, n}^{p(\beta-1)}\in \h
\end{equation*}
where  $v_{L,n}=\min\{v_{n}, L\}$. 
Let us observe that, since $\gamma$ is an increasing function, then it holds
\begin{align*}
(a-b)(\gamma(a)- \gamma(b))\geq 0 \quad \mbox{ for any } a, b\in \R.
\end{align*}
Define the functions 
\begin{equation*}
\Lambda(t)=\frac{|t|^{p}}{p} \quad \mbox{ and } \quad \Gamma(t)=\int_{0}^{t} (\gamma'(\tau))^{\frac{1}{p}} d\tau. 
\end{equation*}
Fix $a, b\in \R$ such that $a>b$. Then, from the above definitions and applying Jensen's inequality we get
\begin{align*}
\Lambda'(a-b)(\gamma(a)-\gamma(b)) &=(a-b)^{p-1} (\gamma(a)-\gamma(b))= (a-b)^{p-1} \int_{b}^{a} \gamma'(t) dt \\
&= (a-b)^{p-1} \int_{b}^{a} (\Gamma'(t))^{p} dt \geq \left(\int_{b}^{a} (\Gamma'(t)) dt\right)^{p}.
\end{align*}
In similar way, we can prove that the above inequality is true for any $a\leq b$. This means that
\begin{equation}\label{Gg}
\Lambda'(a-b)(\gamma(a)-\gamma(b))\geq |\Gamma(a)-\Gamma(b)|^{p} \mbox{ for any } a, b\in\R. 
\end{equation}
By \eqref{Gg}, 
\begin{align}\label{Gg1}
|\Gamma(v_{n})(x)- \Gamma(v_{n})(y)|^{p} \leq |v_{n}(x)- v_{n}(y)|^{p-2} (v_{n}(x)- v_{n}(y))((v_{n}v_{L,n}^{p(\beta-1)})(x)- (v_{n}v_{L,n}^{p(\beta-1)})(y)). 
\end{align}
Now, we take $\gamma(v_{n})=v_{n} v_{L, n}^{p(\beta-1)}$ as test function in \eqref{Pn}. Then, in view of \eqref{Gg1}, 
\begin{align}\label{BMS}
&[\Gamma(v_{n})]^{p}_{s,p}+\int_{\R^{N}} V(\e_{n} x)|v_{n}|^{p}v_{L, n}^{p(\beta-1)} dx \nonumber \\
&\leq \iint_{\R^{2N}} \frac{ |v_{n}(x)- v_{n}(y)|^{p-2} (v_{n}(x)- v_{n}(y))}{|x-y|^{N+sp}} ((v_{n}v_{L, n}^{p(\beta-1)})(x)-(v_{n} v_{L,n}^{p(\beta-1)})(y)) \,dx dy +\int_{\R^{N}} V(\e_{n} x)v_{n}^{p}v_{L,n}^{p(\beta-1)} dx \nonumber\\
&=\int_{\R^{N}} \left(\frac{1}{|x|^{\mu}}*F(v_{n}) \right)f(v_{n}) v_{n} v_{L,n}^{p(\beta-1)} dx.
\end{align}
Since $\Gamma(v_{n})\geq \frac{1}{\beta} v_{n} v_{L,n}^{\beta-1}$, 
and invoking Theorem \ref{Sembedding}, we get 
\begin{equation}\label{SS1}
[\Gamma(v_{n})]^{p}_{s,p}\geq S_{*}^{-1} |\Gamma(v_{n})|^{p}_{\p}\geq \left(\frac{1}{\beta}\right)^{p} S_{*}^{-1} |v_{n} v_{L,n}^{\beta-1}|^{p}_{\p}.
\end{equation}
On the other hand, from the boundedness of $(v_{n})$ and Lemma \ref{lemK}, it follows that there exists $C_{0}>0$ such that
\begin{equation}\label{GbAY}
\sup_{n\in \mathbb{N}} |K(v_{n})|_{\infty}\leq C_{0}.
\end{equation}
From the growth assumptions on $f$, for any $\xi>0$ there exists $C_{\xi}>0$ such that
\begin{equation}\label{SS2}
|f(v_{n})|\leq \xi |v_{n}|^{p-1}+C_{\xi}|v_{n}|^{\p-1}.
\end{equation}
Choosing $\xi\in (0, V_{0})$, and using \eqref{SS1}, \eqref{GbAY} and \eqref{SS2}, we can see that \eqref{BMS} yields
\begin{align}\label{simo1}
|w_{L,n}|^{p}_{\p}\leq C\beta^{p} \int_{\R^{N}} |v_{n}|^{\p} v_{L,n}^{p(\beta-1)} dx,
\end{align}
where $w_{L,n}:=v_{n} v_{L,n}^{\beta-1}$.
Now, we take $\beta=\frac{\p}{p}$ and fix $R>0$. Observing that $0\leq v_{L,n}\leq v_{n}$, we can deduce that
\begin{align}\label{simo2}
\int_{\R^{N}} v^{\p}_{n}v_{L,n}^{p(\beta-1)}dx&=\int_{\R^{N}} v^{\p-p}_{n} v^{p}_{n} v_{L,n}^{\p-p}dx \nonumber\\
&=\int_{\R^{N}} v^{\p-p}_{n} (v_{n} v_{L,n}^{\frac{\p-p}{p}})^{p}dx \nonumber \\
&\leq \int_{\{v_{n}<R\}} R^{\p-p} v^{\p}_{n} dx+\int_{\{v_{n}>R\}} v^{\p-p}_{n} (v_{n} v_{L,n}^{\frac{\p-p}{p}})^{p}dx \nonumber\\
&\leq \int_{\{v_{n}<R\}} R^{\p-p} v^{\p}_{n} dx+\left(\int_{\{v_{n}>R\}} v^{\p}_{n} dx\right)^{\frac{\p-p}{\p}} \left(\int_{\R^{N}} (v_{n} v_{L,n}^{\frac{\p-p}{p}})^{\p}dx\right)^{\frac{p}{\p}}.
\end{align}
Since $v_{n}\rightarrow v$ in $W^{s, p}(\R^{N})$, we can see that for any $R$ sufficiently large
\begin{equation}\label{simo3}
\left(\int_{\{v_{n}>R\}} v^{\p}_{n} dx\right)^{\frac{\p-p}{\p}}\leq \frac{1}{2C \beta^{p}}.
\end{equation}
Putting together \eqref{simo1}, \eqref{simo2} and \eqref{simo3} we get
\begin{equation*}
\left(\int_{\R^{N}} (v_{n} v_{L,n}^{\frac{\p-p}{p}})^{\p} \, dx\right)^{\frac{p}{\p}}\leq C\beta^{p} \int_{\R^{N}} R^{\p-p} v^{\p}_{n} dx<\infty
\end{equation*}
and taking the limit as $L\rightarrow \infty$, we obtain $v_{n}\in L^{\frac{(\p)^{2}}{p}}(\R^{N})$.

Now, using $0\leq v_{L,n}\leq v_{n}$ and passing to the limit as $L\rightarrow \infty$ in \eqref{simo1}, we have
\begin{equation*}
|v_{n}|_{\beta\p}^{\beta p}\leq C \beta^{p} \int_{\R^{N}} v^{\p+p(\beta-1)}_{n} \, dx,
\end{equation*}
from which we deduce that
\begin{equation*}
\left(\int_{\R^{N}} v^{\beta\p}_{n} dx\right)^{\frac{1}{(\beta-1)\p}}\leq (C \beta)^{\frac{1}{\beta-1}} \left(\int_{\R^{N}} v^{\p+p(\beta-1)}_{n}\, dx\right)^{\frac{1}{p(\beta-1)}}.
\end{equation*}
For $m\geq 1$ we define $\beta_{m+1}$ inductively so that $\p+p(\beta_{m+1}-1)=\p \beta_{m}$ and $\beta_{1}=\frac{\p}{p}$. Then
\begin{equation*}
\left(\int_{\R^{N}} v_{n}^{\beta_{m+1}\p} dx\right)^{\frac{1}{(\beta_{m+1}-1)\p}}\leq (C \beta_{m+1})^{\frac{1}{\beta_{m+1}-1}} \left(\int_{\R^{N}} v_{n}^{\p\beta_{m}}\, dx\right)^{\frac{1}{\p(\beta_{m}-1)}}.
\end{equation*}
Let us define
$$
D_{m}=\left(\int_{\R^{N}} v_{n}^{\p\beta_{m}}\, dx\right)^{\frac{1}{\p(\beta_{m}-1)}}.
$$
Using an iteration argument, we can find $C_{0}>0$ independent of $m$ such that
$$
D_{m+1}\leq \prod_{k=1}^{m} (C \beta_{k+1})^{\frac{1}{\beta_{k+1}-1}}  D_{1}\leq C_{0} D_{1}.
$$
Taking the limit as $m\rightarrow \infty$ we get $|v_{n}|_{\infty}\leq K$ for all $n\in \mathbb{N}$.
Moreover, from Corollary $5.5$ in \cite{IMS}, we can deduce that $v_{n}\in \mathcal{C}^{0, \alpha}(\R^{N})$ for some $\alpha>0$ (independent of $n$) and $[v_{n}]_{\mathcal{C}^{0, \alpha}(\R^{N})}\leq C$, with $C$ independent of $n$. Since $v_{n}\rightarrow v$ in $W^{s, p}(\R^{N})$, we can infer that $\lim_{|x|\rightarrow\infty}v_{n}(x)=0$ uniformly in $n\in \mathbb{N}$.
Moreover, using $v_{n}\rightarrow v\neq 0$ in $W^{s,p}(\R^{N})$ again, it follows that there are $(y_{n})$ and $\beta, R>0$ such that
$$
\int_{\B_{R}(y_{n})}|v_{n}|^{p}\,dx\geq \beta.
$$
If $|v_{n}|_{\infty}\rightarrow 0$, then we have
$$
\beta\leq \int_{\B_{R}(y_{n})}|v_{n}|^{p}\,dx\leq |\B_{R}(0)| |v_{n}|^{p}_{\infty}\rightarrow 0
$$
which gives a contradiction.
Hence, there exists $\sigma>0$ such that $|v_{n}|_{\infty}\geq \sigma$ for all $n\in \mathbb{N}$.
\end{proof}

At this point, we are able to conclude the proof of Theorem \ref{thmA1} as follows.
Let $u_{\e_{n}}$ be a solution to \eqref{R}. Then $v_{n}=u_{\e_{n}}(\cdot+\tilde{y}_{n})$ solves the problem
\begin{equation*}
\left\{
\begin{array}{ll}
(-\Delta)^{s}_{p}v_{n} + V_{n}(x) |v_{n}|^{p-2}v_{n}= \left(\frac{1}{|x|^{\mu}}*F(v_{n})\right)f(v_{n}) &\mbox{ in } \R^{N}\\
v_{n}\in W^{s,p}(\R^{N}), \quad v_{n}>0 &\mbox{ in } \R^{N}, 
\end{array}
\right.
\end{equation*}
where $V_{n}(x)= V(\e_{n}x+ \e_{n}\tilde{y}_{n})$, and $(\tilde{y}_{n})\subset \R^{N}$ is given by Lemma \ref{prop3.3}. Moreover, up to a subsequence, $y_{n}:=\e_{n}\tilde{y}_{n}\rightarrow y\in M$ and $v_{n}\rightarrow v$ in $W^{s,p}(\R^{N})$.
If $p_{n}$ denotes a maximum point of $v_{n}$, then there exists $R>0$ such that $|p_{n}|\leq R$ for all $n\in \mathbb{N}$, by Lemma \ref{lemMoser}. Thereby the maximum point of $u_{\e_{n}}$ is given by $z_{\e_{n}}=p_{n}+\tilde{y}_{n}$ and we get $\e_{n}z_{\e_{n}}\rightarrow y\in M$. From the continuity of $V$ we deduce that $V(\e_{n}z_{\e_{n}})\rightarrow V(y)=V_{0}$ as $n\rightarrow \infty$. Therefore, if $u_{\e}$ is a solution to \eqref{R} then $w_{\e}(x)=u_{\e}(x/\e)$ is a solution to \eqref{P}. Hence, the maximum points $\eta_{\e}$ and $z_{\e}$ of $w_{\e}$ and $u_{\e}$, respectively, satisfy $\eta_{\e}=\e z_{\e}$ from which we can infer that $V(\eta_{\e_{n}})\rightarrow V_{0}$ as $n\rightarrow \infty$.

Finally, we give an estimate on the decay properties of solutions $w_{n}$ of \eqref{P}. 
For this aim, using Lemma $7.1$ in \cite{DPQna}, we can find a positive function $\omega$ and a constant $C>0$ such that for large $|x|>R_{0}$ it holds that $\omega(x)\leq \frac{C}{1+|x|^{N+sp}}$ and $(-\Delta)^{s}_{p}\omega+\frac{V_{0}}{2}\omega^{p-1}\geq 0$.
From Lemma \ref{lemMoser}, \eqref{GbAY} and \eqref{SS2} we obtain for some large $R_{1}>0$
\begin{align*}
(-\Delta)^{s}_{p}v_{n}+\frac{V_{0}}{2}v_{n}^{p-1}&=\left(\frac{1}{|x|^{\mu}}*F(v_{n})\right)f(v_{n})-\left(V_{n}-\frac{V_{0}}{2}\right)v_{n}^{p-1} \\
&\leq \left(\frac{1}{|x|^{\mu}}*F(v_{n})\right)f(v_{n})-\frac{V_{0}}{2}v_{n}^{p-1} \\
&\leq 0 \quad \mbox{ for } |x|>R_{1}.
\end{align*}
In view of the continuity of $v_n$ and $\omega$, there exists some constant $C_{1}>0$ such that 
$\psi_{n}:=v_{n}-C_{1}w\leq 0$ on $|x|=R_{2}$, where $R_{2}=\max\{R_{0}, R_{1}\}$. Moreover, 
arguing as in Remark $3$ in \cite{AI}, we can prove that $(-\Delta)^{s}_{p}\psi_{n}+\frac{V_{0}}{2}\psi_{n}^{p-1}\leq 0$ holds for $|x|\geq R_{2}$. Thus, by the maximum principle \cite{DPQ}, we infer that $\psi_{n}\leq 0$ for $|x|\geq R_{2}$, that is $v_{n}\leq C_{1} w$ for $|x|\geq R_{2}$. Recalling that $w_{n}(x)=u_{n}(\frac{x}{\e_{n}})=v_{n}(\frac{x}{\e_{n}}-\tilde{y}_{n})$ we can deduce that 
\begin{align*}
w_{n}(x)&=u_{n}\left(\frac{x}{\e_{n}}\right)=v_{n}\left(\frac{x}{\e_{n}}-\tilde{y}_{n}\right) \\
&\leq \frac{\tilde{C}}{1+|\frac{x}{\e_{n}}-\tilde{y}_{n}|^{N+sp}} \\
&=\frac{\tilde{C} \e_{n}^{N+sp}}{\e_{n}^{N+sp}+|x- \e_{n} \tilde{y}_{n}|^{N+sp}} \\
&\leq \frac{\tilde{C} \e_{n}^{N+sp}}{\e_{n}^{N+sp}+|x-\eta_{\e_{n}}|^{N+sp}}
\end{align*}
which gives the required estimate.

\end{document}